\documentclass[11pt]{amsart}
\usepackage{}
\usepackage{amssymb}

\usepackage{amsfonts}
\usepackage{amsmath}

\input xy
\xyoption{all}

\usepackage{amsfonts}
\usepackage{mathrsfs}

\usepackage{latexsym}
\usepackage{graphicx}
\usepackage{amscd,amssymb,amsmath,amsbsy,amsthm}
\usepackage[all]{xy}
\usepackage[citecolor=red,colorlinks,plainpages,urlcolor=blue]{hyperref}
\usepackage{verbatim}

\usepackage{tikz}
\usetikzlibrary{arrows,calc}
\tikzset{
>=stealth',
help lines/.style={dashed, thick},
axis/.style={<->},
important line/.style={thick},
connection/.style={thick, dotted},
}

\theoremstyle{definition}
\newtheorem{theorem}{Theorem}[section]
\newtheorem{lemma}[theorem]{Lemma}
\newtheorem{corollary}[theorem]{Corollary}
\newtheorem{prop}[theorem]{Proposition}

\theoremstyle{definition}
\newtheorem{definition}[theorem]{Definition}
\newtheorem{example}[theorem]{Example}
\newtheorem{remark}[theorem]{Remark}

\newtheorem{thm}{Theorem}

\newenvironment{thmbis}[1]
  {\renewcommand{\thethm}{\ref{#1}}
   \addtocounter{thm}{0}%
   \begin{thm}}
  {\end{thm}}

 \DeclareMathOperator{\gr}{gr}

\DeclareMathOperator{\Hom}{{Hom}} 
\DeclareMathOperator{\Ext}{{Ext}}

 \DeclareMathOperator{\End}{End}

\DeclareMathOperator{\Coh}{Coh}
\DeclareMathOperator{\Mod}{Mod\hbox{-}}
\DeclareMathOperator{\Gm}{\mathbb{G}_m}

\DeclareMathOperator{\Gr}{Gr}

\newcommand{\inj}{\hookrightarrow}

\def\angl#1{{\{ #1\}}}

\newcommand{\fs}{\mathfrak{s}}
\newcommand{\fl}{\mathfrak{l}}
\newcommand{\fg}{\mathfrak{g}}
\newcommand{\calC}{\mathcal{C}}
\newcommand{\bbC}{\mathbb{C}}
\newcommand{\bbG}{\mathbb{C}}
\newcommand{\calD}{\mathbb{D}}
\newcommand{\calA}{\mathcal{A}}
\newcommand{\bbQ}{\mathbb{Q}}
\newcommand{\calM}{\mathcal{M}}
\newcommand{\calP}{\mathcal{P}}
\newcommand{\calQ}{\mathcal{Q}}
\newcommand{\fh}{\mathfrak{h}}
\newcommand{\calO}{\mathcal{O}}

\newcommand{\calE}{\mathcal{E}}
\newcommand{\calF}{\mathcal{F}}
\newcommand{\fE}{\mathfrak{E}}
\newcommand{\fF}{\mathfrak{F}}
\newcommand{\bbZ}{\mathbb{Z}}
\newcommand{\calV}{\mathcal{V}}
\newcommand{\bbP}{\mathbb{P}}

\begin{document}

\title{Categorification via blocks of modular representations for $\mathfrak{sl}_n$}
\author{Vinoth Nandakumar}
\address{School of Mathematics and Statistics, University of Sydney, NSW 2006 Australia}
\email{vinoth.nandakumar@sydney.edu.au (or vinoth.90@gmail.com)}

\author{Gufang Zhao}
\address{Department of Mathematics and Statistics, University of Massachusetts,  Amherst, MA 01003 U.S.A.}
\curraddr{The University of Melbourne,
School of Mathematics and Statistics,
Parkville VIC 3010,
Australia}
\email{gufangz@unimelb.edu.au}
\dedicatory{Dedicated to our friend, Dmitry Vaintrob}

\begin{abstract} 
Bernstein, Frenkel, and Khovanov have constructed a categorification of tensor products of the standard representation of $\mathfrak{sl}_2$, where they use singular blocks of category $\mathcal{O}$ for $\mathfrak{sl}_n$ and translation functors. Here we construct a positive characteristic analogue using blocks of representations of $\fs\fl_n$ over a field $\textbf{k}$ of characteristic $p$ with zero Frobenius character, and singular Harish-Chandra character. We show that the aforementioned categorification admits a Koszul graded lift, which is equivalent to a geometric categorification constructed by Cautis, Kamnitzer, and Licata using coherent sheaves on cotangent bundles to Grassmanians. In particular, the latter admits an abelian refinement. With respect to this abelian refinement, the stratified Mukai flop induces a perverse equivalence on the derived categories for complementary Grassmanians. This is part of a larger project to give a combinatorial approach to Lusztig's conjectures for representations of Lie algebras in positive characteristic. 
\end{abstract}

\maketitle

\section{Introduction}
In \cite{bfk}, Bernstein, Frenkel and Khovanov categorify the action of $\fs\fl_2$ on the tensor product $(\mathbb{C}^2)^{\otimes n}$ using singular blocks of category $\mathcal{O}$ for $\mathfrak{sl}_n$. In \cite{FKK}, Frenkel, Kirillov and Khovanov show that the classes of the simple objects in these representation categories match up with the dual canonical basis in $(\mathbb{C}^2)^{\otimes n}$, specialized at $q=1$. These results can be used to give a combinatorial approach to the Kazhdan-Lusztig Conjectures in type A, and categorical techniques have since been used widely in representation theory. In the present paper, we extend this approach to representation categories of Lie algebras in positive characteristic. We show that the resulting categorification can be equipped with a Koszul grading, using the theory of geometric categorical actions developed by Cautis, Kamnitzer, and Licata and geometric localization theory developed by Bezrukavnikov, Mirkovi\'c, and Rumynin. In \S~\ref{further}, we discuss some open questions and give a summary of the sequel \cite{nzh2} based on the techniques developed in the present paper. 

Categorification refers to the idea of lifting algebraic, and representation theoretic, structures and maps to the categorical level. In particular, given a linear map between two vector spaces, the vector spaces are lifted to categories, and the linear map is lifted to a functor. In many cases, including the one discussed in this paper, and the one studied by Bernstein, Frenkel, Kirillov and Khovanov  \cite{bfk}, the categories in question themselves arise in representation theoretic contexts, and interesting properties of these representation categories can be deduced using the general framework of categorification. One of the first examples was given by Ariki \cite{ariki} and Grojnowski \cite{groj}, which says the highest weight modules of $\widehat{\mathfrak{sl}_n}$ can be categorified using suitable representation blocks of affine Hecke algebras and their quotients. Chuang and Rouquier \cite{cr} use the categorical framework that they developed to prove Broue's abelian defect conjecture. Categorical techniques also play an important role in Elias-Williamson's proof of Kazhdan-Lusztig conjectures for category $\mathcal{O}$. Recently categorification has been applied to representation theory in positive characteristic; in particular, see the work of Riche and Williamson \cite{ew} on $p$-canonical bases for algebraic groups, and work of the first author in collaboration with Rina Anno \cite{an} and David Yang \cite{ny} on representations of $\mathfrak{sl}_n$ with two-row nilpotent Frobenius characters. 

We are interested in $\fs\fl_2$-categorifications. An $\mathfrak{sl}_2$-representation on a finite-dimensional complex vector space $V$ consists of a weight space decomposition $V = \bigoplus_{r \in \mathbb{Z}} V_r$, linear maps $E_{r+1}: V_r \rightarrow V_{r+2}$ and $F_{r+1}: V_{r+2} \rightarrow V_r$, such that $$E_{r-1} F_{r-1} - F_{r+1} E_{r+1} = r \cdot \text{Id}$$ Loosely speaking, when we \textit{categorify} the representation $V$, we replace each weight space $V_r$ by a category $\mathcal{C}_r$ such that $K^0(\mathcal{C}_r) \simeq V_r$; and replace the maps $E_{r+1}$ and $F_{r+1}$ by functors $\mathcal{E}_{r+1}: \mathcal{C}_r \rightarrow \mathcal{C}_{r+2}$ and $\mathcal{F}_{r+1}: \mathcal{C}_{r+2} \rightarrow \mathcal{C}_{r}$ which satisfy the categorical $\mathfrak{sl}_2$ relation (as spelled out in (\ref{catrel}) in Section \ref{subsec:sl2}). We will also need Chuang-Rouquier's notion of an $\fs\fl_2$-categorification, which consists of some extra data: endomorphisms $X \in \text{End}(\bigoplus \mathcal{E}_r), T \in \text{End}(\bigoplus \mathcal{E}_{r+2} \circ \mathcal{E}_r)$ satisfying certain compatibilities (see Section~\ref{subsec:sl2}).

One of the first interesting examples of this was given by Bernstein, Frenkel, and Khovanov \cite{bfk} (and motivated by the geometric constructions from Beilinson, Lusztig, MacPherson \cite{blm}, and Grojnowski \cite{groj}). The $\mathfrak{sl}_2$-representation in question is $V = (\mathbb{C}^2)^{\otimes n}$, which has a weight space decomposition $V = \bigoplus_{0 \leq i \leq n} V_{-n+2i}$ with $\text{dim}(V_{-n+2r}) = \binom{n}{r}$. The category $\mathcal{C}_{-n+2r}$ is taken to be the singular block of category $\mathcal{O}$ for $\mathfrak{sl}_n$ with Harish-Chandra character $\mu_r = - \rho + e_1 + \cdots + e_r$ (here $\rho$ is the half-sum of all positive roots, and $e_1, \cdots, e_n$ are coordinates on the Cartan matrix). $\mathcal{E}_{-n+2r+1}$ (resp. $\mathcal{F}_{-n+2r+1}$) are given by translation functors between these blocks, and are given by tensoring with $\mathbb{C}^n$ (resp. $(\mathbb{C}^n)^*$) followed by projection. A basis for $K^0(\mathcal{C}_{-n+2r})$ is then given by the classes of the Verma modules, and this basis is identified with the standard basis of the weight space $V_{-n+2r}$. Chuang and Rouquier \cite{cr} verify that this is a $\fs\fl_2$-categorification in their sense.

In this paper, we first construct a modular analogue of this result, using blocks of representations of the Lie algebra $\fg:=\fs\fl_n$ defined over an algebraically closed field $\textbf{k}$ of characteristic $p > n$; see Sections~\ref{modularrep} and \ref{subsec:const_state} for more details. We will be categorifying the same representation $V=(\mathbb{C}^2)^{\otimes n}$, and will take $\mathcal{C}_{-n+2r}$ to be the category $\text{Mod}_{0, \mu_r}(\textbf{U}\fg)$ of finitely generated $\textbf{U}\fg$-modules, on which the Harish-Chandra center acts via the same generalized central character $\mu_r$, and the Frobenius center acts by central character $0$. As before, these central characters are carefully chosen so that the rank of the Grothendieck group is equal to $\binom{n}{r}$. Again, the functors $\mathcal{E}_{-n+2r+1}$ and $\mathcal{F}_{-n+2r+1}$ are translation functors between the corresponding blocks, and are given by tensoring with $\textbf{k}^n$ (resp. $(\textbf{k}^n)^*$) followed by projection. 

Here we emphasize that although the $\mathfrak{sl}_2$-representation being categorified is defined over $\mathbb{C}$, the categories used in this construction are representation categories of $\mathfrak{sl}_n$ over a field of positive characteristic. This is closely related to the categorification constructed by Chuang and Rouquier \cite[\S~7.5]{cr} using representations of $\textbf{SL}_n(\textbf{k})$ (see Remark~\ref{cr7.5}). A more precise statement is given in Section~\ref{subsec:const_state}. 

\begin{thm}\label{ThmIntA} 
Consider the categories $\calC_{-n+2r}=\text{Mod}_{0,\mu_r}(\textbf{U}\fg)$, for $0 \leq r \leq n$, and the translation functors $\mathcal{E}_{-n+2r+1}$, $\mathcal{F}_{-n+2r+1}$. There exist (explicitly constructed) functors and natural transformations,  which satisfy the conditions of a $\fs\fl_2$-categorification in the sense of Chuang and Rouquier. 
\end{thm}




The second main result of the present paper is on a graded lift of the categorification from Theorem~\ref{ThmIntA}, which is equivalent to one constructed by Cautis, Kamnitzer, and Licata  \cite{ckl}; see Section~\ref{statementB} for more details. When working in the graded setting, following Rouquier \cite{Rou}, the analogous notion categorifying an action of the quantum group $U_q(\mathfrak{sl}_2)$ is that of a ``strong categorical  $\mathfrak{sl}_2$-action'', recalled in Section~\ref{subsec:sl2} (see  also \cite[\S~2.1]{ckl}). In the construction  of Cautis, Kamnitzer, and Licata \cite{ckl}, the category $\calC_{-n+2r}$ is taken to be the derived category of $\Gm$-equivariant coherent sheaves on $T^* \Gr(r,n)$ (here $\Gr(r,n)$ is the Grassmannian of $r$-dimensional vector spaces in $\textbf{k}^n$, and $\Gm$ is the multiplicative group). The functors that are denoted by $\textbf{E}(-n+2r+1)$ and $\textbf{F}(-n+2r+1)$ are given by certain pull-push maps using an intermediary space. They show that this fits into their framework of ``geometric categorical $\mathfrak{sl}_2$-actions'', and consequently deduce that this is a strong $\mathfrak{sl}_2$-categorification.  Upon taking Grothendieck groups, one obtains the $U_q(\mathfrak{sl}_2)$-representation $V^{\otimes n}$, where $V$ is the standard representation of $U_q(\mathfrak{sl}_2)$.

The graded lift of the modular representation categories in question, denoted by $\text{Mod}^{\text{fg, gr}}_{0,\mu_r} (\textbf{U}\fg)$, is called the Koszul grading,  constructed by Riche  \cite{riche}, using a localization equivalence that builds upon the framework developed by Bezrukavnikov, Mirkovi\'c, and Rumynin in \cite{bmr}. The localization equivalence yields that: $$ D^b \text{Coh}_{\Gm}(T^* \Gr(r,n)) \simeq D^b \text{Mod}^{\text{fg, gr}}_{0, \mu_r}(\textbf{U} \mathfrak{g}) .$$ 
The graded lift of the categorification from Theorem~\ref{ThmIntA} is obtained from 
the  categorification  of  Cautis, Kamnitzer, and Licata \cite{ckl} via twisting by certain line bundles, followed by applying the above equivalences to it. See Section~\ref{statementB} for a more precise version of the below. 

\begin{thm}\label{IntrB}
On the categories $\calD_{-n+2r}=\text{Mod}^{\text{fg,gr}}_{0,\mu_r} (\textbf{U}\fg)$ there are functors $E_{-n+2r+1}, F_{-n+2r+1}$ and an octuple of natural transforms making it 
 a strong $\mathfrak{sl}_2$-categorification.  Moreover, this categorification is equivalent to the categorification constructed by Cautis, Kamnitzer, and Licata \cite{ckl}. The forgetful functor $F:\text{Mod}^{\text{fg,gr}}_{0,\mu_r} (\textbf{U}\fg)\to \text{Mod}^{\text{fg}}_{0,\mu_r} (\textbf{U}\fg)$ commutes with this categorification and the one from Theorem~\ref{ThmIntA}. 
\end{thm} 

\noindent{\bf Corollary.}
The geometric categorification constructed by Cautis, Kamnitzer, and Licata \cite{ckl} admits an abelian refinement.
\bigskip

The construction of Cautis, Kamnitzer, and Licata  \cite{ckl} uses derived categories, and the existence of an abelian refinement is not immediate from the definition. A more general theory of abelian refinements and exotic t-structures is developed by Cautis and Koppensteiner  \cite[ Corollary 9.2]{cautis}. Our proof of Theorem \ref{IntrB} involves studying the image of the $\mathfrak{sl}_2$ functors, under the linear Koszul duality equivalence that is used  by Riche\cite{riche}. This leads us to study coherent sheaves on the Grothendieck-Springer varieties $\tilde{\mathfrak{g}}_{\mathcal{P}}$. The category of coherent sheaves studied here is closely related to Cautis and Kamnitzer's categorical loop $\mathfrak{sl}_n$ action on these categories \cite{ck16}. See Remark \ref{refine} for a more precise discussion of these connections. 

An explicit derived equivalence $D^b\Coh(T^*\Gr(k,N)) \cong D^b\Coh(T^*\Gr(N-k,N))$ is achieved as a consequence of the categorification by Cautis, Kamnitzer, and Licata \cite{ckl}, which answered an earlier open question of Kawamata and Namikawa on whether stratified Mukai flops induced derived equivalences. As an application of the above theorem, we prove that the derived equivalences induced by these stratified Mukai flops are perverse equivalences in the sense of Chuang and Rouquier \cite{CR2}. More precisely, we have the following (namely, we can describe explicitly the change of $t$-structures under this class of stratified Mukai flops). See \S~\ref{subsec:perv} and Remark~\ref{rmk:3.14} for a more detailed discussion.

\medskip
\noindent{\bf Corollary.} 
The derived equivalence $D^b\Coh(T^*\Gr(k,N)) \cong D^b\Coh(T^*\Gr(N - k,N))$ of Cautis, Kamnitzer, and Licata \cite{ckl} is a perverse equivalence, when both sides are endowed with the $t$-structures coming from the localization of Riche \cite{riche}.
\bigskip




\subsection*{Organization of this paper} 
In Section \ref{thmAproof}, we recall some background material about categorification, and about modular representations of Lie algebras. We then state in more details and prove Theorem~\ref{ThmIntA}. In Section~\ref{thmBproof}, we recall some background material about Riche's localization results; then we state in more details and prove Theorem~\ref{IntrB}. In Section \ref{further}, we discuss some open problems and further directions.  

\subsection*{Acknowledgements}
We are very much indebted to Roman Bezrukavnikov, who first suggested to us that the categorification \cite{ckl} admits an abelian refinement using exotic $t$-structures and linear Koszul duality. We would like to thank Joel Kamnitzer, Jonathan Brundan, Sabin Cautis, Alexander Kleshchev, Simon Riche, Anthony Licata, Rapha\"el Rouquier, Mikhail Khovanov, Will Hardesty, Michael Ehrig, Oded Yacobi, David Yang, You Qi and Kevin Coulembier for helpful discussions. We are grateful to Jim Humphreys for helpful comments on a previous version of this manuscript. The first author would also like to thank the University of Utah (in particular, Peter Trapa), and the University of Sydney (in particular, Gus Lehrer and Ruibin Zhang) for supporting this research.
During the revision of this paper, the second named author was affiliated with the Institute of Science and Technology Austria, Hausel Group, supported by
the Advanced Grant Arithmetic and Physics of Higgs moduli spaces No. 320593 of the European Research Council.

\section{Construction of the categorification}
\label{thmAproof}
In this section, we construct the $\fs\fl_2$-categorification using blocks of representations of $\fs\fl_n$, and prove  Theorem~\ref{ThmIntA}. In Sections~\ref{subsec:sl2} and \ref{modularrep}, we collect some background material about categorification and modular representations of Lie algebras, that will be used in later sections. The reader may wish to return to these two sections after reading Section \ref{subsec:const_state}, and readers familiar with the material may wish to skip Sections~\ref{subsec:sl2} and \ref{modularrep}. 
\subsection{Categorical $\mathfrak{sl}_2$-actions} \label{subsec:sl2}
We give an overview of categorical $\mathfrak{sl}_2$-actions, and state some of Chuang and Rouquier's results \cite{cr}; this sub-section is purely expository.

\begin{definition} \label{weaksl2} A weak $\mathfrak{sl}_2$-categorification
is the data of an adjoint pair $(E,F)$ of exact endo-functors on a $\textbf{k}$-linear abelian category $\calA$, such that
\begin{enumerate}
\item the action of $e = [E]$ and $f = [F]$ on $V = \bbQ\otimes K(\calA)$ gives a locally-finite $\mathfrak{sl}_2$-representation.
\item $\calA$ admits a $t$-structure the heart of which is
artinian and noetherian, and
the classes of the simple objects are weight vectors;
\item $F$ is isomorphic to a left adjoint of $E$.
\end{enumerate} 
\end{definition}

Let $V = \oplus_{r \in \mathbb{Z}} V_r$ be the weight decomposition, let $\calA_{r}$ be the full subcategory consisting of objects in $\calA$ whose classes in the Grothendieck group  lie in $V_{r}$. One easy consequence of this definition is that we have a decomposition $\calA = \bigoplus_{r \in \mathbb{Z}} \calA_{r}$  \cite[Lemma~5]{cr}.

\begin{definition} \label{strongsl2} An $\mathfrak{sl}_2$-categorification is a weak $\mathfrak{sl}_2$-categorification, together with the data of $X \in \text{End}(E)$ and $T \in \text{End}(E^2)$, and $q \in \textbf{k}^{\times}, a \in \textbf{k}$, such that: \begin{itemize} \item $(1_E T) \circ (T 1_E) \circ (1_E T) = (T 1_E) \circ (1_E T) \circ (T 1_E)$ in $\text{End}(E^3)$ \item $(T + 1_{E^2}) \circ (T - q 1_{E^2}) = 0$ in $\text{End}(E^2)$ \item In $\text{End}(E^2)$, we have: $$T \circ (1_E X) \circ T = 
\begin{cases} q X 1_E \qquad \qquad \text{if $q \neq 1$ }
\\ X 1_E - T \qquad  \text{ if $q = 1$  }
\end{cases}$$
 \item $X-a$ is locally nilpotent \end{itemize} \end{definition}
 
The categorical $\mathfrak{sl}_2$-relation  
\begin{equation} \label{catrel} EF|_{\mathcal{A}_{-r}} \oplus \text{Id}_{\mathcal{A}_{-r}}^{\oplus r} \simeq FE|_{\mathcal{A}_{-r}}, \qquad EF|_{\mathcal{A}_r} \simeq FE|_{\mathcal{A}_r} \oplus \text{Id}_{\mathcal{A}_r}^{\oplus r} \end{equation} 
(here $r \geq 0$) follows \cite[Theorem~5.27]{cr} as a consequence of Definition~\ref{strongsl2}.
Roughly speaking, a $\mathfrak{sl}_2$-categorification is a collection of categories $\mathcal{A}_r$; adjoint functors between them $$E_{r+1}: \mathcal{A}_r \rightarrow \mathcal{A}_{r+2}, F_{r+1}: \mathcal{A}_{r+2} \rightarrow \mathcal{A}_r$$ that satisfy the categorical $\mathfrak{sl}_2$-relation. The results of Chuang and Rouquier imply that the notion of  $\mathfrak{sl}_2$-categorification satisfies this property. 

Below we briefly recall the definition of a \textit{strong categorical $\mathfrak{sl}_2$-action}; for sake of brevity, we omit some of the details and refer the reader to \cite[\S~2.1]{ckl} for a list of the compatibilities that must be satisfied. This data is essentially equivalent to a representation of the categorical quantum group as defined by Rouquier \cite{Rou} and Khovanov-Lauda \cite{khovlauda1}, \cite{khovlauda2} (see also \cite[Remark 2.2]{ckl}, and \cite{cl}, \cite{brundan} for related work). Given these pieces of data, one obtains an action of the quantum group $U_q(\mathfrak{sl}_2)$ on the split Grothendieck group  \cite[\S~2.2]{ckl11}; thus one may think of this as a categorification of a representation of the quantum group for $\mathfrak{sl}_2$. 

\begin{definition} \label{2LieAlgebra} A \textit{strong categorical $\mathfrak{sl}_2$-action} consists of the following pieces of data: \begin{itemize} 
\item For each $-n \leq r \leq n$, a $\textbf{k}$-linear, $\mathbb{Z}$-graded additive category $\mathcal{D}(r)$. 
Let $\mathcal{D}(r)=0$ if $r \notin [-n, n]$. 
Here graded means that each $D(r)$ has a shift functor $\langle1\rangle$ which is an equivalence.
\item  For each $k \geq 1$, functors $E^{(k)}(r): \mathcal{D}(r-k) \rightarrow \mathcal{D}(k+r)$ and $F^{(k)}(r): \mathcal{D}(k+r) \rightarrow \mathcal{D}(r-k)$ Denote $E(r):=E^{(1)}(r), F(r):=F^{(1)}(r)$; we refer to $E^{(k)}(r), F^{(k)}(r)$ as the divided powers. 
An octuple of natural transforms $(\eta_1,\eta_2,\epsilon_1,\epsilon_2, \iota, \pi, \hat{T}(r),\hat{X}(r))$:
\item Adjunction morphisms: $$ \eta_1: \text{id} \rightarrow F^{(k)}(r) E^{(k)}(r) \langle rk \rangle, \qquad \eta_2: \text{id} \rightarrow E^{(k)}(r) F^{(k)}(r) \langle -rk \rangle $$ $$ \epsilon_1: F^{(k)}(r) E^{(k)}(r) \rightarrow \text{id} \langle rk \rangle, \qquad \epsilon_2: E^{(k)}(r) F^{(k)}(r) \rightarrow \text{id} \langle -rk \rangle $$ \item Morphisms: $$ \iota: E^{(k+1)}(r) \langle k \rangle \rightarrow E(k+r) E^{(k)}(r-1), \; \; \; \pi: E(k+r) E^{(k)}(r-1) \rightarrow E^{(k+1)}(r) \langle -k \rangle $$ \item Morphisms: $$ \hat{X}(r): E(r) \langle -1 \rangle \rightarrow E(r) \langle 1 \rangle, \qquad \hat{T}(r): E(r+1) E(r-1) \langle 1 \rangle \rightarrow E(r+1) E(r-1) \langle -1 \rangle$$ \end{itemize} 
These  subject to  compatibility conditions \cite[\S~2.1]{ckl}. \footnote{Here $r$ is what has been denoted by $\lambda$ in \cite{ckl}, and $k$ here is what  has been denoted by $r$ in {\it loc. cit.}.}
\end{definition} 
The notion of a ``geometric categorical $\mathfrak{sl}_2$-action'' has been developed by Cautis, Kamitzer, and Licata\cite{ckl}, which is more convenient when working in the framework of coherent sheaves, and they show that it can be used to construct a categorical $\mathfrak{sl}_2$-action. To avoid repetition, we omit the definition \cite[Definition~2.2, \S~4, \S~5]{ckl}. 

\subsection{Modular representations of Lie algebras} \label{modularrep} Let $G$ be a semisimple, simply connected, algebraic group, with Lie algebra $\mathfrak{g}$, defined over a field $\textbf{k}$ of characteristic $p$. Assume that $p$ satisfies conditions (H1)-(H3)  \cite[B.6]{jantzen2}. In the case that we will be studying, where $G=SL_n(\textbf{k})$, $\mathfrak{g}=\mathfrak{sl}_n(\textbf{k})$, it is sufficient that $p>n$. Let $\mathfrak{g} = \mathfrak{n}^- \oplus \mathfrak{h} \oplus \mathfrak{n}^+$ be the triangular decomposition,  $W$ be the associated Weyl group, and $\rho$ the half-sum of all positive roots. Recall that we have the twisted action of $W$ on $\mathfrak{h}^*$: $$w \cdot \lambda = w(\lambda + \rho) - \rho$$
 In this subsection we will collect some facts about the representation theory of $\mathfrak{g}$, and refer the reader to Jantzen's expository article \cite{jantzen2} for a detailed treatment. 

First we will need the following description of the center of the universal enveloping algebra $\textbf{U} \mathfrak{g}$. Define the Harish-Chandra center $Z_{\text{HC}}$ to be $Z_{\text{HC}} = (\textbf{U}\mathfrak{g})^G$. Given an element $x \in \mathfrak{g}$, it is known that there exists a unique $x^{[p]} \in \mathfrak{g}$ such that $x^p - x^{[p]} \in Z(\textbf{U} \mathfrak{g})$. Then the Frobenius center $Z_{\text{Fr}}$ is defined to be the subalgebra generated by $\{ x^p - x^{[p]} \; | \; x \in \mathfrak{g} \}$. In fact, for $p >n$, $Z(\textbf{U} \mathfrak{g})$ is generated by $Z_{\text{Fr}}$ and $Z_{\text{HC}}$ \cite[\S~C]{jantzen2}. 

\begin{definition} 
For any $\nu\in\fh^*$, its orbit under the Weyl group dot-action $\nu \in \mathfrak{h}^* /\!/ W$ defines a character of the Harish-Chandra character $Z_{\text{HC}}$.  (We follow the convention that the Harish-Chandra character are $\rho$-shifted.) Let $\text{Mod}^{\text{fg}}_{0, \nu}(\textbf{U} \mathfrak{g})$ be the full category of all finitely generated modules, where the Frobenius center $Z_{\text{Fr}}$ acts with a fixed zero character, and  $Z_{\text{HC}}$ acts via a generalized central character $\nu \in \mathfrak{h}^* /\!/ W$. We will refer to $\text{Mod}^{\text{fg}}_{0, \nu}(\textbf{U} \mathfrak{g})$ as a {\it block}. The {\it restricted Lie algebra} $\textbf{U}_0\mathfrak{g}$ is the quotient of $\textbf{U}\mathfrak{g}$ by the ideal $\langle x^p - x^{[p]} \; | \; x \in \mathfrak{g} \rangle$, and $\text{Rep}(\textbf{U}_0 \mathfrak{g})$ be the category of all finitely generated representations of $\textbf{U}_0 \mathfrak{g}$. \end{definition}
\begin{definition} 
A central character $\nu \in \mathfrak{h}^*$ is said to be regular if its (twisted) $W$-orbit contains $|W|$ elements, and singular otherwise. \end{definition} 
The following decomposition \cite[Section C]{jantzen2} is why we refer to the categories $\text{Mod}^{\text{fg}}_{0, \nu}(\textbf{U} \mathfrak{g})$ as {\it blocks}; we will refer to it as a singular (or regular) block depending on whether $\nu$ is singular (or regular). The principal block is the one containing the trivial (i.e. one-dimensional) module. $$\text{Rep}(\textbf{U}_0 \mathfrak{g}) \simeq \bigoplus_{\nu \in \mathfrak{h}^*/\!/W} \text{Mod}^{\text{fg}}_{0, \nu}(\textbf{U} \mathfrak{g})$$ 
\begin{definition} 
Let $\mu \in \mathfrak{h}^*$ be an integral weight, i.e., comes from $\Lambda=\Hom(T,\bbG_m)$. 
Define $\textbf{U}_0\mathfrak{b}$ to be the quotient of $U\mathfrak{b}$ by the ideal $\langle x^p - x^{[p]} \; | \; x \in \mathfrak{b} \rangle$, and $\textbf{k}_{\mu}$ the $\textbf{U}_0\mathfrak{b}$-module with highest weight $\mu$. Then the baby Verma module $Z(\mu)$ is defined as: $$ Z(\mu) = \textbf{U}_0\mathfrak{g} \otimes_{\textbf{U}_0\mathfrak{b}} \textbf{k}_{\mu}$$ 
\end{definition} 
It is defined analogously to its counterparts in category $\mathcal{O}$, and they share many properties. Most importantly, the module $Z(\mu)$ has unique  simple quotient, which we denote $L(\mu)$. Thus for any integral weight $\mu \in \mathfrak{h}^*$, we have a simple module $L(\mu)$; these are pairwise non-isomorphic, any simple module in $\text{Rep}(\textbf{U}_0 \mathfrak{g})$ occurs in this way, and  $L(\mu)$ lies in $\text{Mod}^{\text{fg}}_{0, \nu}(\textbf{U} \mathfrak{g})$ precisely if $\mu \in W \cdot \nu$. In other words, the irreducible objects in $\text{Mod}^{\text{fg}}_{0, \nu}(\textbf{U} \mathfrak{g})$ are $\{ L(w \cdot \nu) \; | \; w \in W \}$. The number of simple objects in $\text{Mod}^{\text{fg}}_{0, \nu}(\textbf{U} \mathfrak{g})$ is hence equal to the size of the $W$-orbit of $\nu$ (using the twisted Weyl group action) \cite[Sections C and D]{jantzen2}. 
\begin{definition} Given a dominant integral weight $\lambda \in \Lambda^+=\{ \langle \lambda ,\alpha \rangle \geq 0 \hbox{ for each positive root } \alpha\}$, let $\gamma = - w_0 \lambda$. Here $w_0 \in W$ is the long word in the Weyl group, and let $\mathcal{O}(\gamma)$ be the corresponding line bundle on the flag variety. Then the Weyl module $\textbf{V}(\lambda)$ is defined to be the following $G$-module: $$\textbf{V}(\lambda) = H^0(\mathcal{O}(\gamma))^*$$ \end{definition}
Weyl modules can also be defined over the Lie algebra $\mathfrak{g}$ using the below equivalence, involving the first Frobenius kernel $G_1$. See   \cite[Chapter 7]{jantzen} for a definition of the subscheme $G_1 \subset G$ and   \cite[Chapter 5]{hum06} for a proof of the equivalence below $$\text{Rep}(G_1) \simeq \text{Rep}(\textbf{U}_0 \mathfrak{g}).$$ \begin{definition} 
The Weyl module $V(\lambda) \in \text{Rep}(U_0 \mathfrak{g})$ is defined by first taking the image of $\textbf{V}(\lambda)$ under the restriction map: $\text{Rep}(G) \rightarrow \text{Rep}(G_1)$, and then take the image under the above equivalence: $\text{Rep}(G_1) \simeq \text{Rep}(\textbf{U}_0 \mathfrak{g})$. 
\end{definition} 
Intuitively, the Weyl modules $V(\lambda)$ may be thought of as the reduction $V_{\lambda} \otimes_{\mathbb{Z}} \textbf{k}$, where $V_{\lambda}$ is the irreducible module in characteristic zero, equipped with a $\mathbb{Z}$-form (to make this precise, one needs the notion of a minimal admissible lattice  \cite[Section 1.1]{lattice}. (See  \cite[Section 2.13]{jantzen2} and Humphreys' expository note \cite{weyl} for more details about Weyl modules.)

Let $\overline{\lambda} \in \mathfrak{h}^*$ be the image of $\lambda \in \Lambda$ after reducing modulo $p$. 
Then the module $V(\lambda)\in \text{Rep}(U_0 \mathfrak{g})$ has a unique maximal submodule, and $L(\overline{\lambda})$  is the quotient  \cite[Section 2.14]{jantzen}. Further, $V(\lambda)$ lies in the same block as $L(\overline{\lambda})$ (this follows from the alternative approach to the Weyl modules using minimal admissible lattices).

\begin{remark} For readers who are more familiar with the category $\text{Rep}(G)$, it is worth noting that any  irreducible representation of $\textbf{U}_0 \mathfrak{g}$ can be lifted to an irreducible representation of $G$ with a restricted highest weight (i.e. a weight $\lambda \in \Lambda_{\mathbb{Z}}$ with $0 \leq \langle \lambda, \check{\alpha} \rangle < p$ for each simple $\alpha$). More precisely, under the equivalence $\text{Rep}(G_1) \simeq \text{Rep}(\textbf{U}_0 \mathfrak{g})$, the irreducible $L(\overline{\lambda})$ corresponds to the restriction of the irreducible module for $G$ with highest weight $\lambda$, under the restriction map $\text{Rep}(G) \rightarrow \text{Rep}(G_1)$. All other irreducibles in $\text{Rep}(G)$ may be obtained using the Steinberg tensor product theorem. In this paper, we will solely be dealing with representations of the restricted enveloping algebra $\textbf{U}_0 \mathfrak{g}$; more specifically, the blocks $\text{Mod}^{\text{fg}}_{0, \nu}(\textbf{U} \mathfrak{g})$. \end{remark}

\subsection{Statement of Theorem~\ref{ThmIntA}} \label{subsec:const_state}
Let $\mathfrak{g} = \mathfrak{sl}_{n}$ be defined over an algebraically closed field $\textbf{k}$ of characteristic $p$, with $p>n$.

Our goal is to categorify the $\mathfrak{sl}_2$-action on $V = (\mathbb{C}^2)^{\otimes n}$. Denote its weight spaces, and the restrictions of the operators $E$ and $F$ between them, as follows (here $V_{-n+2r} = 0$ if $r \notin \{ 0, 1, \cdots, n \}$): \begin{align*} V = (\mathbb{C}^2)^{\otimes n} = \bigoplus_{r=0}^{n} V_{-n + 2r}, \qquad & \text{dim} (V_{-n+2r}) = {n \choose r} \end{align*} \begin{align*} E = \bigoplus_{r=0}^n E_{-n+2r+1}, \qquad &E_{-n+2r+1}: V_{-n+2r} \rightarrow V_{-n+2r+2} \\ F = \bigoplus_{r=0}^n F_{-n+2r+1}, \qquad & F_{-n+2r+1}: V_{-n+2r+2} \rightarrow V_{-n+2r}. \end{align*}

Let $e_i$ be the diagonal matrices whose $(i,i)$th entry is 1 and zeros otherwise, so that the positive roots of $\fg$ are given by $\{ e_i - e_j \; | \; i < j \}$, and the simple roots are $e_i - e_{i+1}$ for $1 \leq i \leq n-1$. Recall that $\rho$ may be expressed as follows: $$ \rho = \frac{n-1}{2} e_1 + \frac{n-3}{2} e_2 + \cdots + \frac{1-n}{2} e_n. $$ Recall that the fundamental weights are $\lambda_i = e_1 + \cdots + e_i$, for $1 \leq i \leq n-1$. For $0 \leq r \leq n$, let us define $$\mu_r = -\rho + e_1 + e_2 + \cdots + e_r, \qquad \mathcal{C}_{-n+2r} = \text{Mod}^{\text{fg}}_{0, \mu_r}(\textbf{U} \mathfrak{g})$$
 
From the results of Section \ref{modularrep}, we see that the Grothendieck group of $\mathcal{C}_{-n+2r}$ has rank ${{n}\choose{r}}$, {\it i.e.}, equal to the dimension of the vector space $V_{-n+2r}$. Let $\textbf{k}^n$ be the standard representation of $\mathfrak{sl}_n$, and by $(\textbf{k}^n)^{*}$ its dual, respectively. Let us define:
\begin{align}\label{eqn:trans_func} \textbf{E}_{-n+2r+1}: \mathcal{C}_{-n+2r} \rightarrow \mathcal{C}_{-n+2r+2}, \qquad & \textbf{E}_{-n+2r+1}(M) = \text{proj}_{\mu_{r+1}}(M \otimes \textbf{k}^n); \\ \textbf{F}_{-n+2r+1}: \mathcal{C}_{-n+2r+2} \rightarrow \mathcal{C}_{-n+2r}, \qquad & \textbf{F}_{-n+2r+1}(N) = \text{proj}_{\mu_r}(N \otimes (\textbf{k}^n)^*). \notag
\end{align} 
Here $\text{proj}_\mu$  for any $\mu\in\fh/\!/W$ is the functor taking the direct summand on which the Harish-Chandra center acts by $\mu$.
\begin{remark} These functors are analogous to those of Bernstein, Frenkel, and Khovanov  \cite[Theorem 1, Corollary 1]{bfk}. Instead of the singular blocks of category $\mathcal{O}$ used there, we use singular blocks of modular representations of $\mathfrak{sl}_n$. The singular Harish-Chandra characters we use are the same as  used by Bernstein, Frenkel, and Khovanov \cite{bfk}. \end{remark}
In the below, note that although we are categorifying a representation of $\mathfrak{sl}_2(\mathbb{C})$, the categories used in this construction are representation categories of $\mathfrak{sl}_n$ over a field of positive characteristic. 
\begin{thmbis}{ThmIntA} \label{theorema} 
Let $\mathcal{C}_{-n+2r}:=\text{Mod}_{0,\mu_r}(\textbf{U}\fg)$, and let $\textbf{E}_{-n+2r+1}$ and $\textbf{F}_{-n+2r+1}$ be the functors defined in \eqref{eqn:trans_func}. There exist natural transformations satisfying the conditions of a strong $\fs\fl_2$-categorification as in Definition~\ref{strongsl2}. 

On $\bigoplus_{r}K^0(\calC_{-n+2r})$,  the functors \eqref{eqn:trans_func} recovers the action of $\mathfrak{sl}_2$ on $(\mathbb{C}^2)^{\otimes n}$. 
\end{thmbis}
As a corollary, one obtains the categorical $\mathfrak{sl}_2$-relation, in accordance with \cite[Theorem 5.27]{cr}:
$$ \textbf{F}_{-n+2r+1} \circ \textbf{E}_{-n+2r+1} \oplus \text{Id}^{\oplus r} \simeq \textbf{E}_{-n+2r-1} \circ \textbf{F}_{-n+2r-1} \oplus \text{Id}^{\oplus n-r} $$

\begin{remark} \label{cr7.5} In  \cite[Section~7.5]{cr}, Chuang and Rouquier construct an $\mathfrak{sl}_2$ categorification using rational representations of $G = \text{SL}_n(\textbf{k})$. Using the equivalence $\text{Rep}(\textbf{U}_0(\mathfrak{g})) \simeq \text{Rep}(G_1)$, we can restate our categorification using blocks of $\text{Rep}(G_1)$. Under the natural restriction map $\text{Rep}(G) \rightarrow \text{Rep}(G_1)$, one can show that the two categorifications are compatible, {\it i.e.}, the restriction maps commute with the translation functors between the blocks. We also expect that Theorem A can be extended to higher Frobenius kernels. Let $G_r$ be the kernel of the $r$-th Frobenius map, $F^r: G \rightarrow G$. Our proof relies on manipulations with Weyl modules, which also works in this level of generality. It would be interesting to compute the $\mathfrak{sl}_2$-representation obtained by counting the number simple objects in the Grothendieck groups of the categories corresponding to the weight spaces. It is straightforward to check that the two categorifications would be compatible, i.e. that the restriction maps commute with translation functors between the blocks. \end{remark}

\begin{remark} \label{K0} In Bernstein, Frenkel, and Khovanov's categorification \cite{bfk}, one has an explicit identification between the Grothendieck groups of the categories involved (singular blocks of category $\mathcal{O}$) and the weight spaces that they categorify. The classes of the Verma modules correspond to the standard basis in these weight spaces. In our setting, we do not have an explicit identification, although Theorem A implies that an identification exists. The baby Verma modules all have the same class in the Grothendieck group, so it is not clear which modules categorify the standard basis. This issue is fixed in the sequel to this paper (see Section 3.2 of \cite{nzh2}) by looking at a graded version of these categories. In that section we also calculate the images of the baby Verma modules under the translation functors $\textbf{E}_{-n+2r+1}$ and $\textbf{F}_{-n+2r+1}$ in the Grothendieck group; this is a straightforward adaptation of Propositions 6 and 7 in Section 3.1 of \cite{bfk}. See also Section \ref{further} for a related discussion. \end{remark} 

\subsection{Perverse equivalence}
\label{subsec:perv}
The main result of Chuang and Rouquier \cite{cr} applied to this setting yields that we have the existence of a derived equivalence \[\Phi_r:D^b(\calC_{-n+2r})
\to
D^b(\calC_{n-2r}),\] lifting
the action of $\exp(-F) \exp(E) \exp(-F): K^0(\calC_{-n+2r})\to K^0(\calC_{n-2r})$. Furthermore, the behavior of this derived equivalence with respect to the natural $t$-structures is controlled by a perversity function as follows. 

Let $S$ be the set of simple objects of $\calC_r$. We define two filtrations of $S$:
\[S_i = \{V \in S\mid  F_{i+1}V = 0\}\hbox{ and } S'_i = \{V \in S \mid E_{i+1}V = 0\}.\]
These filtrations on $S$ induced filtrations on  $\calC_r$ 
by Serre subcategories, denoted by $\calC_{r,i}$ and $\calC_{r,i}'$ respectively. Let $D_{\calC_{r,i}}(\calC_r)$ denote the thick subcategory
of $D^b(\calC_r)$ of complexes with cohomology in $\calC_{r,i}$; similar for $D_{\calC_{r,i}'}(\calC_r)$.

The following is a direct consequence of Theorem~\ref{theorema}  and \cite[Proposition~8.4]{CR2}.
\begin{corollary}\label{cor:2.13}
For any $r$ and  $i$, $\Phi_r[i]$
restricts to an equivalence $D^b_{\calC_{-n+2r,i}}(\calC_{-n+2r})\to D^b_{\calC_{n-2r,i}}(\calC_{n-2r})$, and the induced equivalence between quotient triangulated
categories \[D^b_{\calC_{-n+2r,i}}(\calC_{-n+2r})/D^b_{\calC_{-n+2r,i+1}}(\calC_{-n+2r})\cong D^b_{\calC_{n-2r,i}}(\calC_{n-2r})/D^b_{\calC_{-n+2r,i+1}}(\calC_{-n+2r})\] restricts to an equivalence \[\calC_{-n+2r,i}/\calC_{-n+2r,i+1}\cong \calC_{n-2r,i}/\calC_{-n+2r,i+1}.\]
\end{corollary}

Following the terminologies of Chuang and Rouquier \cite{CR2}, the above equivalence is a perverse equivalence, with respect to perversity function $p(i)=i$. 
In particular, for two abelian categories $A$ and $A'$ and a perverse equivalence $F:D^b(A)\to D^b(A')$ with respect to a perversity function $p$, the abelian category $A'$ is determined by $ A$ and $p$ and vice versa \cite[Proposition 4.17]{CR2}.

\subsection{Outline of the proof}
To prove Theorem \ref{theorema}, first we will show that this gives a weak $\mathfrak{sl}_2$-categorification as in Definition~\ref{weaksl2}. This is a consequence of the following two results, the first of which is well-known, and the second is a combinatorial calculation. Note there are infinitely many Weyl modules $V(\lambda)$ lying in each block. This lemma implies that one can choose a subset that constitutes a basis. 

\begin{lemma} \label{span} The Grothendieck group of $\mathcal{C}_{-n+2r}$ is spanned by the classes of the Weyl modules which lie in it. \end{lemma} 

\begin{prop} \label{calc} Suppose that $V(\lambda)$ lies in $\mathcal{C}_{-n+2r}$. We have the following equality in the Grothendieck group: $$ [\{ \textbf{F}_{-n+2r+1} \circ \textbf{E}_{-n+2r+1} \oplus \text{Id}^{\oplus r} \} V(\lambda)] \simeq [ \{ \textbf{E}_{-n+2r-1} \circ \textbf{F}_{-n+2r-1} \oplus \text{Id}^{\oplus n-r} \} V(\lambda)] $$ \end{prop} 
The proofs of Lemma~\ref{span} and Proposition~\ref{calc} can be found in Section~\ref{subsec:proofA}.

To complete the proof of Theorem A, we need to construct endomorphisms $X \in \text{End}(E)$ and $T \in \text{End}(E^2)$, satisfying the compatibilities from Section \ref{subsec:sl2}. For this purpose,  we follow Chuang and Rouquier \cite[\S~7.4]{cr}, where they show that the categorification constructed by Bernstein, Frenkel, and Khovanov satisfies these properties.  

\subsection{Proof of Theorem~\ref{theorema}} 
\label{subsec:proofA}
Lemma \ref{span} is certainly not new, and known to experts. For the reader's convenience, we briefly sketch a proof.

\begin{proof}[Proof of Lemma \ref{span}] First, note that the classes of the Weyl modules in $\text{Rep}(G)$ span $K^0(\text{Rep}(G))$. To see this, the classes of the irreducibles form a basis, and the transition matrix between the irreducibles and Weyl modules is upper triangular, using the standard highest weight theory. Note however that this upper triangularity argument breaks down after passing to $\text{Rep}(\textbf{U}_0 \mathfrak{g})$.

Recall from Section \ref{modularrep} that there is a natural restriction map $\text{Rep}(G) \rightarrow \text{Rep}(G_1)$. Each simple object in $\text{Rep}(G_1)$ is the image of a simple object in $\text{Rep}(G)$  \cite[Section 2.2]{nakano}. Hence the restriction map is surjective on the level of Grothendieck groups. Therefore, $\text{Rep}(G_1)$ is spanned by the classes of all (restrictions of) Weyl modules. We also have $$\text{Rep}(G_1) \simeq \text{Rep}(\textbf{U}_0 \mathfrak{g})$$ hence all Weyl modules span the Grothendieck group of $\text{Rep}(\textbf{U}_0 \mathfrak{g})$. Lemma \ref{span} follows now using the fact that $\text{Rep}(\textbf{U}_0 \mathfrak{g})$ splits up as a direct sum of blocks, and that each Weyl module lies in one of these blocks. \end{proof}

\begin{proof}[Proof of Proposition \ref{calc}] Since $V(\lambda)$ lies in $\mathcal{C}_{-n+2r}$, $\overline{\lambda} \sim {\mu}_{r}$; thus, for some $w \in W$: $$ \overline{\lambda} = e_{w(1)} + \cdots + e_{w(r)} - \rho $$ Let us assume that each of the parts of $n$ parts of $\lambda$ has size at least $1$ (if not, simply add one box to each row of $\lambda$; this does not change $V(\lambda)$ as an $\mathfrak{sl}_n$-module). Let us refer to the rows indexed by $w(1), \cdots, w(r)$ as being ``marked'', and all other rows as being ``unmarked''. Let $S(\lambda) \subset \Lambda_{\mathbb{Z}}^+$ be the set of weights whose partitions are obtained from $\lambda$ by adding exactly one box. Let $T(\lambda) \subset \Lambda_{\mathbb{Z}}^+$ be the set of weights whose partitions are obtained from $\lambda$ by removing a box in some row. Recall the following two identities: \begin{align*} [V(\lambda) \otimes V] &= \sum_{\mu \in S(\lambda)} [ V(\mu) ] \\ [V(\lambda) \otimes V^*] &= \sum_{\mu \in T(\lambda)} [ V(\mu) ]. \end{align*} These are well-known in the characteristic zero setting. In our set-up, note that the identities hold in $K^0(\text{Rep}(G))$ since the character map is an isomorphism onto the Grothendieck group. Using the aforementioned functor $\text{Rep}(G) \rightarrow \text{Rep}(G_1) \simeq \text{Rep}(\textbf{U}_0 \mathfrak{g})$, these two formulas follow. By the definitions of $\textbf{E}_{-n+2r+1}, \textbf{F}_{-n+2r-1}$, we have: \begin{align*} [\textbf{E}_{-n+2r+1} (V(\lambda))] &= \sum_{\mu \in S(\lambda), \; \overline{\mu} \sim \mu_{r+1}} [V(\mu)] \\ [\textbf{F}_{-n+2r-1} (V(\lambda))] &= \sum_{\mu \in T(\lambda), \; \overline{\mu} \sim \mu_{r-1}} [V(\mu)] \end{align*} Suppose that  $\overline{\mu} \sim \mu_{r+1}$; then for some $w' \in S_n$, we have that:  \begin{align*}  \overline{\mu} + \rho &= w'(\mu_{r+1} + \rho) = w'(e_1 + \cdots + e_{r+1}) \\ \iff \overline{\mu} &= e_{w'(1)} + \cdots + e_{w'(r+1)} - \rho 
 \end{align*} It follows that for a given $\mu \in S(\lambda)$, we have that $\overline{\mu} \sim \mu_{r+1}$ precisely if the box being added doesn't lie in one of the $k$ marked rows. Similarly, a given $\mu \in T(\lambda)$ satisfies $\overline{\mu} \sim \mu_{r-1}$ precisely if a box being removed lies in one of the $r$ marked rows. Let us say that an unmarked (resp. marked) row is {\it typical} if it's possible to add (resp. remove) a box in that row; let us call a row {\it atypical} if it isn't typical. Let $Q(\lambda) \subset \Lambda^+$ consist of all weights whose partitions are obtained from $\lambda$ by adding a box in an typical, unmarked row, and removing a box in a typical, marked row. Then we have that: \begin{equation} \label{comb1} 
 [\textbf{F}_{-n+2r+1} \textbf{E}_{-n+2r+1} (V(\lambda))] = c_1 [V(\lambda)] + \sum_{\mu \in Q(\lambda)} [V(\mu)]; \qquad \text{($c_1$ is the number of typical, unmarked rows)}; 
 \end{equation} 
 \begin{equation}  \label{comb2} 
 [\textbf{E}_{-n+2r-1} \textbf{F}_{-n+2r-1} (V(\lambda))] =c_2 [V(\lambda)] + \sum_{\mu \in Q(\lambda)} [V(\mu)]; \qquad \text{($c_2$ is the number of typical, marked rows)}. 
 \end{equation} 
It remains to check  $c_1 + r = c_2 + (n-r)$, or equivalently that there is a bijection between atypical marked rows, and atypical unmarked rows. This follows from the observation that the $i$-th row is marked and atypical precisely if the $i+1$-st row is unmarked and atypical. To see this, if $i$-th row is marked, and atypical, then $\lambda_{i+1} = \lambda_i$. But using the explicit formula for $\overline{\lambda}$ and $\rho$, $\overline{\lambda_{i+1}} - \overline{\lambda_i} = 0,1$ or $2$, with the first case only occuring when $i \in \{ w(1), \cdots, w(r) \}$, and $i+1 \notin \{ w(1), \cdots, w(r) \}$. Hence the $(i+1)$-st row is unmarked and also atypical; the converse follows similarly.  \end{proof}
 
\begin{example} Let $\mathfrak{g} = \mathfrak{sl}_5, r=2$ and $p=7$. The weight $\lambda = (20,13,7,2,2)$ satisfies $\overline{\lambda} \sim \mu_{2}$, since $\overline{\lambda} = e_1 + e_4 - \rho$ and $\mu_{2} = e_1 + e_2 - \rho$. Then: \begin{align*} [\textbf{E}_{0}(V(\lambda))] &= [V(20, 13, 8, 2, 2)] + [V(20, 14, 7, 2, 2)] \\ [\textbf{F}_{-1}(V(\lambda))] &= [V(19, 13, 7, 2, 2)]. \end{align*} It is straightforward to verify equations (\ref{comb1}) and (\ref{comb2}). Here $c_1=2, c_2=1$ and $$ Q(\lambda) = \{ (19, 13, 8, 2, 2), (19, 14, 7, 2, 2) \} \qquad \qquad \qquad \qquad. \blacksquare $$ \end{example}

So far we have constructed a ``weak $\mathfrak{sl}_2$-categorification'' (see Definition \ref{weaksl2}). To complete the proof of Theorem \ref{theorema}, it suffices to construct $X \in \text{End}(\textbf{E})$ and $T \in \text{End}(\textbf{E}^2)$ that satisfy the conditions specified in Definition \ref{strongsl2} when $q=1$ and $a=0$; the functorial $\mathfrak{sl}_2$ relations then follow \cite[Theorem 5.27]{cr} (here $\textbf{E} = \bigoplus_{0 \leq r \leq n} \textbf{E}_{-n+2r+1}$). To do this, we closely follow the template of Chuang and Rouquier \cite[\S~7.4]{cr}, where they show that the categorification constructed by Bernstein, Frenkel, and Khovanov  \cite{bfk} is in fact an $\mathfrak{sl}_2$-categorification; only very minor modifications are needed. 

\begin{proof}[Proof of Theorem~\ref{theorema}] Given a module $M \in \text{Mod}^{\text{fg}}_{0, \mu}(\textbf{U} \mathfrak{g})$ for some $\mu$, define $\overline{X}_M \in \text{End}_{\mathfrak{g}}(V \otimes M)$ by noting that $\text{Hom}(V \otimes V^* \otimes M, M) = \text{End}(V \otimes M)$, and using the adjoint map $\mathfrak{g} \otimes M \rightarrow M$. More explicitly, $\overline{X}_M(v \otimes m) = \Omega(v \otimes m)$, where $\Omega = \sum_{i,j=1}^n e_{ij} \otimes e_{ji}$. Define also $\overline{T}_M \in \text{End}_{\mathfrak{g}}(V \otimes V \otimes M)$ by simply switching the factors: $T_M (v_1 \otimes v_2 \otimes m) = v_2 \otimes v_1 \otimes m$. This gives endomorphisms $\overline{X}$ (resp. $\overline{T}$) of the functor $V \otimes -$ (resp. $V \otimes V \otimes -$). 

Let us explain why the endomorphism $\overline{X}$ descends to an endomorphism $X \in \text{End}(\textbf{E})$ \cite[\S~7.4.3]{cr}. It is sufficient to identify $\textbf{E}$ as a generalised eigenspace of $\overline{X}$ acting on $V \otimes -$. This will in turn follow once we verify that $\textbf{E} Z(\lambda)$ is a generalised eigenspace of $\overline{X}$ acting on $V \otimes Z(\lambda)$, where $\lambda$ is a weight lying in the shifted $W$-orbit of $\mu$ (since any object is mapped onto surjectively by an extension of baby Vermas). If $C$ is the central Casimir element, and $\delta: \textbf{U}\mathfrak{g} \rightarrow \textbf{U}\mathfrak{g}\otimes \textbf{U}\mathfrak{g}$ denotes the comultiplication, we have that: \begin{align*} C = \sum_{i,j=1}^n e_{ij} e_{ji} &= \sum_{i=1}^n e_{ii}^2 + \sum_{1 \leq i < j \leq n} (e_{ii} - e_{jj}) + 2 \sum_{1 \leq i < j \leq n} e_{ji} e_{ij} \\ \Omega &= \frac{1}{2} (\delta(C) - C \otimes 1 - 1 \otimes C). \end{align*} Hence $C$ acts on the baby Verma module $Z(\lambda)$ via the scalar $$ b_{\lambda} = \sum_{i=1}^n \lambda_i^2 + \sum_{1 \leq i < j \leq n} (\lambda_i - \lambda_j). $$ Now consider the action of $\Omega$ on the sub-quotient of $V \otimes Z(\lambda)$ isomorphic to $Z(\lambda+e_i)$. Using the above formula for $\Omega$ in terms of $C$, it follows that $\Omega$ acts on this sub-quotient by the scalar $c_{\lambda, i} := \frac{1}{2}(b_{\lambda + e_i} - b_{\lambda} - b_{e_1})$. It is easy to compute that, provided $p>2$, $c_{\lambda, i} = c_{\lambda, j}$ precisely if $\lambda+e_i$ and $\lambda+e_j$ are in the same $W$-orbit (here one must take the $\rho$-shift into account).

The fact that $\overline{T}$ descends to an endomorphism $T \in \text{End}(E^2)$ follows \cite[Lemma~7.21]{cr}, as does the third condition in Definition \ref{strongsl2}.  The first, second and fourth conditions \cite[Definition~5.21]{cr} that we have to check are self-evident. This completes the proof. 
 \end{proof}
 

\subsection{The $\mathfrak{g}=\mathfrak{sl}_2(\textbf{k})$ example.} \label{n=2}

Let us describe the $\mathfrak{g}=\mathfrak{sl}_2(\textbf{k})$ case in detail. Much of what is discussed in this section is well-known, but we include many of the basic facts for the reader's convenience. 

Here there are two irreducible objects in the principal block: the one-dimensional module $L(0)$, and the $p-1$ dimensional module $L(-2)$ (note that in this cases, both of these irreducibles are Weyl modules). Recall that the baby Verma module $Z(0)$ (resp. $Z(-2)$) is generated by a weight vector $v_0$ (resp. $v_{-2}$) satisfying $Hv_0 = 0$ (resp. $Hv_{-2} = -2 v_{-2}$), and has $L(0)$ (resp. $L(-2)$) as its head; see Section \ref{modularrep} for more details. We have two short exact sequences \begin{align*} 0 \rightarrow L(-2) &\rightarrow Z(0) \rightarrow L(0) \rightarrow 0 \\ 0 \rightarrow L(0) &\rightarrow Z(-2) \rightarrow L(-2) \rightarrow 0 \end{align*} Denote the projective cover of $L(0)$ as $P(0)$, and the projective cover of $L(-2)$ as $P(-2)$. The projective covers $P(0)$ and $P(-2)$ both have dimension $2p$, and can be described as follows. Denote the irreducible Steinberg module $Z(-1)$; then $P(-2) = Z(-1) \otimes \textbf{k}^2$, since translation functors map projectives to projectives. 
Using the results from Humphreys \cite[Theorem 3, pg. 4]{humphreys} (see also  \cite[Chapter 10]{hum06}, in particular 10.6) it follows that $P(0)$ is the indecomposable summand of $Z(-1) \otimes Z(-1)$ containing the weight vector $v_{-1} \otimes v_{-1}$ (here $v_{-1} \in Z(-1)$ is the highest weight vector). The element $v_{-1} \otimes v_{-1}$ generates a module isomorphic to $Z(-2)$, and there exists $w \in Z(-1) \otimes Z(-1)$ in the zero weight space, such that $E(w) \in \langle v_{-1} \otimes v_{-1} \rangle$; then $P(0)$ is generated by $v_{-1} \otimes v_{-1}$ and $w$. Thus we have two short exact sequences: \begin{align*} 0 \rightarrow Z(-2) &\rightarrow P(0) \rightarrow Z(0) \rightarrow 0 \\ 0 \rightarrow Z(0) &\rightarrow P(-2) \rightarrow Z(-2) \rightarrow 0 \end{align*} Let $A = \text{End}(P(0) \oplus P(-2))$; then it follows that $\mathcal{C} \simeq \Mod A$.  Here $A$ is an $8$ dimensional algebra, since $$\text{dim}(\text{Hom}(P(i), P(j))) = [P(i) : L(j)] = 2$$ For $i,j \in \{ 0, -2 \}$, we construct $\psi_{i,j}^1, \psi_{i,j}^2 \in \text{Hom}(P(i), P(j))$ such that $\text{Hom}(P(i), P(j)) = \textbf{k} \psi_{i,j}^1 \oplus \textbf{k} \psi_{i,j}^2$. Let $\psi_{i,i}^1 = \text{id}_{P(i)}$, let $\psi_{0,0}^2$ be the composite map $P(0) \rightarrow Z(0) \rightarrow P(0)$, and let $\psi_{-2, -2}^2$ be the composite map $P(-2) \rightarrow Z(-2) \rightarrow P(-2)$. Let $\psi_{0, -2}^1$ be the composite map $P(0) \rightarrow Z(0) \rightarrow P(-2)$, and $\psi_{-2, 0}^1$ be the composite map $P(-2) \rightarrow Z(-2) \rightarrow P(0)$. To construct $\psi_{0, -2}^2$, note that we can lift the composite map $P(0) \rightarrow Z(0) \rightarrow Z(-2)$ to obtain a map $P(0) \rightarrow P(-2)$; $\psi_{0,-2}^2$ is the unique such map satisfying $\psi_{0,-2}^2(w)=F^{p-1}v_{-1} \otimes v_-$ (here $\textbf{k}^2 = \textbf{k} v_+ \oplus \textbf{k} v_-$). Similarly, to construct $\psi_{-2,0}^2$, we can lift the map $P(-2) \rightarrow Z(-2) \rightarrow Z(0)$ to obtain a map $P(-2) \rightarrow P(0)$; $\psi_{-2,0}^2$ is the unique such map satisfying $\psi_{-2,0}^2(v_{-1} \otimes v_+)=F^{p-1}(v_{-1} \otimes v_{-1})$ and $\psi_{-2,0}^2(v_{-1} \otimes v_-) = Fw$. 
\begin{lemma} The relations in $A$ are as follows; here $i,j \in \{ 0, -2 \}, k \in \{ 1, 2 \}$. \begin{itemize} \item $\psi_{i,j}^k \psi_{i', j'}^{k'} = 0$ unless $j' = i$. \item $\psi_{j,j}^1 \psi_{i,j}^k = \psi_{i,j}^k \psi_{i,i}^1 = \psi_{i,j}^k$; and $\psi_{j,j}^2 \psi_{i,j}^k = \psi_{i,j}^k \psi_{i,i}^2 = 0$ unless $k=1$ and $i=j$. \item If $i \neq j$: $\psi_{i,j}^1 \psi_{j,i}^1 = 0$, and $\psi_{i,j}^{k} \psi_{j,i}^{k'} = \psi_{j,j}^2$ unless $k = k' =1$. 
\end{itemize} \end{lemma} \begin{proof} The first relation, and the first part of the second relation, are obvious. The first part of the third relation follows using the definition of the maps $\psi_{i,j}^1, \psi_{j,i}^1$, and the fact that the compositions $Z(0) \rightarrow P(-2) \rightarrow Z(-2)$ and $Z(-2) \rightarrow P(0) \rightarrow Z(0)$ are zero; the same argument implies that $\psi_{j,j}^2 \psi_{i,j}^1 = \psi_{i,j}^1 \psi_{i,i}^2 = 0$ if $i \neq j$, and $(\psi_{i,i}^2)^2=0$. 

For the second part of the third relation, first note that $\psi_{i,j}^1 \psi_{j,i}^2 = \psi_{j,j}^2$ follows from the definitions. The following calculations imply that $\psi_{0,-2}^2 \psi_{-2,0}^2 = \psi_{-2,-2}^2$ and $\psi_{-2,0}^2 \psi_{0,-2}^2 = \psi_{0,0}^2$; one can similarly verify that $\psi_{i,j}^2 \psi_{j,i}^1 = \psi_{j,j}^2$.
\begin{align*} \psi_{0,-2}^2 \psi_{-2,0}^2 (v_{-1} \otimes v_-) &= \psi_{0,-2}^2 (Fw) = F(v_{-1} \otimes v_+) = \psi_{-2,-2}^2 (v_{-1} \otimes v_-) \\ \psi_{-2,0}^2 \psi_{0,-2}^2 (w) &= \psi_{-2,0}^2 (v_{-1} \otimes v_+) = F^{p-1} (v_{-1} \otimes v_{-1}) = \psi_{0,0}^2(w) 
\end{align*} To complete the proof, it remains to finish the second part of the second relation, {\it i.e.} that when $i \neq j$, $\psi_{j,j}^2 \psi_{i,j}^2 = \psi_{i,j}^2 \psi_{i,i}^2 = 0$. This follows from the other relations: for instance, $\psi_{i,j}^2 \psi_{i,i}^2 = \psi_{i,j}^2 \psi_{j,i}^2 \psi_{i,j}^1 = \psi_{j,j}^2 \psi_{i,j}^1 = 0$ (alternatively, one may use the techniques from the previous paragraph). \end{proof}

\begin{example}
Via \S~\ref{thmBproof}, this algebra has a grading that satisfies the Koszul property, which we briefly describe here. This grading is constructed from the derived equivalence with the category of coherent sheaves on $T^*\bbP^1$, with the equivalence given by a tilting vector bundle, which in this case is given explicitly in \cite[Example~1.10.6]{bmr2}. More precisely,  consider the tilting object $\pi^*(\calO\oplus\calO(1))$ with $\pi:T^*\bbP^1\to \bbP^1$ being the projection. The endomorphism algebra of this tilting object can be found in \cite[Example~10.7]{WZ} given by the following quiver with relations
$$\xymatrix{
\bullet_0\ar@/^/@<1ex>[r]^{\alpha_0}\ar@/^/@<3ex>[r]^{\alpha_1}&
\bullet_1\ar@/^/@<1ex>[l]^{\beta_0}\ar@/^/@<3ex>[l]^{\beta_1} }$$
relations: $\alpha_0\beta_1=\alpha_1\beta_0;$
$\beta_0\alpha_1=\beta_1\alpha_0.$
The functor $R\Hom_{T^*\bbP^1}(\pi^*(\calO\oplus\calO(1)), -)$ induces a derived equivalence between modules over the path algebra of the above quiver with relations with $D^b\Coh(T^*\bbP^1)$. Under this equivalence, the abelian subcategory of $D^b\Coh(T^*\bbP^1)$ corresponding to the abelian category of modules over the path algebra is the category of exotic coherent sheaves. The simple objects in this  category of exotic coherent sheaves are $i_*\calO$ and $i_*\calO(-1)[1]$, which correspond respectively to the two vertices of the quiver. The weight 2 $\bbC^*$-action on the cotangent fibers gives a grading of $\End_{T^*\bbP^1}(\pi^*(\calO\oplus\calO(1)))$. 
In terms of quiver with relations, the grading is given such that $\alpha_i$ has degree zero and $\beta_i$ has degree $2$ for $i=0,1$.  
The presentation of the quiver with relation is so that the arrows from the simple object $S_i$ to $S_j$ in the quiver form a basis of $\Ext^1(S_j,S_i)^*$, and the relations form a basis of $\Ext^2(S_j,S_i)^*$ for $i,j=0,1$  \cite[Proposition~7.3]{WZ}. For dimensional reasons, there are no Ext's in degree higher than 2. 
This endomorphism algebra of the tilting object is Koszul dual to the 8-dimensional algebra $A$ above. 
The grading on $A$ is the homological grading of the Ext-algebra of the simple objects, that is, the arrows are in degree 1, and the relations are degree 2.  
\end{example}

Recall that given two rings $R$ and $S$, and an $(S, R)$-bimodule $M$, we have a functor $F_M: R-\text{mod} \rightarrow S-\text{mod}$, given by $F_M(P) = M \otimes_R P$. The adjoint functor $\overline{F_M}: S-\text{mod} \rightarrow R-\text{mod}$ is given by the $(R, S)$-bimodule $M' = \text{Hom}_{R}(M, R)$.

\begin{lemma} This categorification of the $\mathfrak{sl}_2(\mathbb{C})$-action on $\mathbb{C}^2 \otimes \mathbb{C}^2$ uses the three categories $\textbf{k}-\text{mod}$, $A-\text{mod}$ and $\textbf{k}-\text{mod}$, with functors between them: \begin{align*} E_{-1}: \textbf{k}-\text{mod} \rightarrow A-\text{mod}, \; F_{-1}: A-\text{mod} \rightarrow \textbf{k}-\text{mod} \\ E_1: A-\text{mod} \rightarrow \textbf{k}-\text{mod}, \; F_{1}: \textbf{k}-\text{mod} \rightarrow A-\text{mod} \end{align*} The functor $E_{-1}: \textbf{k}-\text{mod} \rightarrow \text{A}-\text{mod}$ corresponds to the $(A, k)$-bimodule $A \psi_{-2,-2}^1$, the functor $F_{-1}: \text{A}-\text{mod} \rightarrow \textbf{k}-\text{mod}$ corresponds to the $(k, A)$-bimodule $\psi_{-2,-2}^1 A$, and $F_1 \simeq E_{-1}, E_1 \simeq F_{-1}$. \end{lemma}
\begin{proof} The functor $E_{-1}$ sends the $1$-dimensional vector space to the projective (left) $A$-module $A \psi_{-2,-2}^1$, so the first part is clear. Using the above fact about adjoints, the functor $F_{-1}$ corresponds to the (right) $A$-module $\text{Hom}_A( A\psi_{-2,-2}^1, A)$, which is isomorphic to $\psi_{-2,-2}^1 A$. \end{proof}


Recall from Remark \ref{K0} that although Theorem A implies the existence of an isomorphism below, as $\mathfrak{sl}_2$-modules, it does not give us an explicit identification. $$\bigoplus_{k=0}^n K^0(\mathcal{C}_{-n+2k}) \simeq (\mathbb{C}^2)^{\otimes n}$$ 

Below we construct such an identification when $n=2$, and describe the basis in the $\mathfrak{sl}_2$-representation $\mathbb{C}^2 \otimes \mathbb{C}^2 = V_{-2} \oplus V_0 \oplus V_2$ coming from the classes of the irreducible objects. Define $w_{-2} = [Z(-1)] \in V_{-2}$, $w_{0}^1 = [L_0], w_0^{2} = [L_s]$ and $w_2 = [Z(-1)] \in V_2$. Then: \begin{align*} 
E w_{-2} &= [Z(-1) \otimes \textbf{k}^2] = 2(w_0^1 + w_0^2) \\ F w_0^1 &= [L_0 \otimes \textbf{k}^2]|_{\mu_1} = 0, F w_0^2 = [L_s \otimes \textbf{k}^2]|_{\mu_1} = w_{-2} \qquad \\ E w_0^1 &= 0, E w_0^2 = w_2, Fw_2 = 2(w_0^1 + w_0^2) \end{align*}  
The above computation allows us to identify $ K^0(\mathcal{C}_{-2} \oplus \mathcal{C}_0 \oplus \mathcal{C}_2) $ with $\mathbb{C}^2 \otimes \mathbb{C}^2$, as $\mathfrak{sl}_2$-representations, so that: $$ w_{-2} \mapsto v_0 \otimes v_0; w_0^1 \mapsto v_0 \otimes v_1 - v_1 \otimes v_0, w_0^2 \mapsto v_0 \otimes v_1; w_2 \mapsto v_1 \otimes v_1 $$ Under this identification, the four simple objects are mapped to the images of the dual canonical bases for the $U_q(\mathfrak{sl}_2)$-representation, specialized at $q=1$ (see the description in \cite[\S~5.2]{fks}). 

\section{A graded lift and equivalence to a geometric categorification}\label{thmBproof}
Here we state and prove a more precise version of Theorem~\ref{IntrB}, the second main result of this paper.  In Sections~\ref{richebackground} and \ref{subsec:Localization} we give a purely expository overview of linear Koszul duality, and Riche's localization results (building on work of Bezrukavnikov, Mirkovi\'c, and Rumynin, \cite{bmr}); these will be used in later sections, and the reader is encouraged to return to it when needed. Throughout this section, we impose the condition that $\text{char}(\textbf{k}) = p >n$. 

\subsection{Linear Koszul duality} \label{richebackground} Here all scheme theoretic objects are defined over the field $\textbf{k}$. 
We omit the definition of a \textit{dg-scheme} $X$, and the category $\text{DGCoh}(X)$ of \textit{dg-coherent sheaves} on this space \cite[\S~1.8]{riche}. 
Given a vector bundle $F$ over $X$ with sheaf of sections $\mathcal{F}$ we consider a coherent sheaf of dg-algebras  $\mathcal{S} = \mathcal{S}_{\mathcal{O}_X}(\mathcal{F}^{\vee})$, endowed with an internal grading so that $\mathcal{F}^{\vee}$ is of bidegree $(2, -2)$ \cite[\S~2.3]{riche}. Here we use bidegree $(i, j)$ to denote the homogeneous piece of  cohomological degree $i$
and internal degree $j$.
The category $\text{DGCoh}^{\text{gr}}(F)$ is the derived category of  bounded above complexes of $\textbf{G}_m$-dg-modules over $\mathcal{S}$ with cohomology that is quasi-coherent, and locally finitely generated \cite[\S~2.3]{riche}. On this category, we have two shifting functors, $[\ ]$ and $\{\}$, corresponding respectively to the shift in cohomological and internal degrees. 

We similarly consider the coherent sheaf of dg-algebras $\mathcal{R} = \mathcal{S}_{\mathcal{O}_X}(\mathcal{F}^{\vee})$, endowed with an internal grading so that $\mathcal{F}^{\vee}$ is of bidegree $(0, -2)$. 
We use $\Gm$ to denote the multiplicative group. Here the subindex $m$ stands for multiplicative. Recall that $\Gm$-equivariance structure is the same as grading. 
Note that the corresponding derived category of bounded above complexes of $\Gm$-dg-modules over $\mathcal{R}$ is naturally equivalent to $D^b\Coh_{\Gm}(F)$, the derived category of equivariant coherent sheaves on $F$ \cite[Lemma~2.3.2]{riche}. Under these conventions,  Riche  \cite[(2.3.5)]{riche} (see also \cite[(1.1.2)]{Rich2}) describes a ``regrading'' equivalence $\xi$, sending $\calM_q^p$ to $\calM_q^{p-q}$: $$\xi: \text{DGCoh}^{\gr}(T^*\calP)\cong D^b \Coh_{\Gm}(T^*\calP)$$

Now we work in the setup of Section~\ref{modularrep}.  Let $P$ be a parabolic subgroup, and let $\mu$ be a weight lying only on those reflection hyperplanes corresponding to the parabolic $P$. Let $\calP=G/P$ be the corresponding partial flag variety, $T^* \calP$ be the cotangent bundle, $\mathfrak{p}$ the corresponding parabolic Lie algebra, and $\mathfrak{u}$ its unipotent radical. Define: $$ \widetilde{\mathfrak{g}}_{\calP} = \{ (X, gP) \in \mathfrak{g}^* \times \calP \; | \; X|_{g. \mathfrak{u}} = 0 \} $$  

We consider the dg-scheme $\widetilde{\fg}_\calP\times_\fg\{0\}$, which is the derived fiber product; let $\text{DGCoh}^{\gr}(\widetilde{\fg}_\calP\times_\fg\{0\})$ be the derived category of $\textbf{G}_m$-dg-modules.
\begin{lemma}[\cite{riche}, Section 10.1]
 \label{Lem:Koszul}
Linear Koszul duality gives an equivalence of categories: \[\kappa:\text{DGCoh}^{\gr}(\widetilde{\fg}_\calP\times_\fg\{0\})\cong \text{DGCoh}^{\gr}(T^*\calP)\] 
\end{lemma}
\begin{proof}
This follows from \cite[Theorem~2.3.10]{riche}, where we take $F=T^*\calP$ and  $E=\fg^*\times\calP$. Then $F^\perp\subseteq E^*$ is isomorphic to $\widetilde{\fg}_\calP$.
\end{proof}

\subsection{Riche's localization results}\label{subsec:Localization}
We have an equivalence \cite[Theorem~3.4.14]{riche}
\begin{equation}
\label{eqn:gamma}
 \tilde{\gamma}_{\mu}^{\mathcal{P}}:  D^b\text{Mod}^{\text{fg}}_{\mu,0} (\textbf{U} \fg) \simeq \text{DGCoh}(\tilde{\fg}_{\mathcal{P}}\times_{\fg}\{0\}).
\end{equation}
By $\tilde{\fg}_{\mathcal{P}}\times_{\fg}\{0\}$, we mean the dg-scheme $\tilde{\fg}_{\mathcal{P}} \cap^R_{\mathfrak{g}^* \times \mathcal{P}} \mathcal{P}$    \cite[Definition 1.8.3]{riche}.

The equivalence~\eqref{eqn:gamma} has a graded version \cite[Theorem 10.3.1]{riche} (see also \cite[Theorem~1.6.7]{bm}). 
More precisely, the algebra $\textbf{U}\fg$ has a suitable completion which admits a Koszul grading;
In particular, the category of finitely generated graded modules $\text{Mod}^{\text{fg, gr}}_{\mu,0} (\textbf{U}\fg)$  \cite[\S~10.2]{riche} is endowed with a derived equivalence
 \[\tilde{\gamma}_{\mu}^{\mathcal{P}}:  D^b\text{Mod}^{\text{fg,gr}}_{\mu,0} (\textbf{U} \fg) \simeq \text{DGCoh}^{\text{gr}}(\tilde{\fg}_{\mathcal{P}}\times_{\fg}\{0\})\] and a forgetful functor 
$Forg: \text{Mod}^{\text{fg,gr}}_{\mu,0} (\textbf{U} \fg)\to \text{Mod}^{\text{fg}}_{\mu,0} (\textbf{U} \fg)$ so that the following diagram commutes 
\[\xymatrix{
D^b\text{Mod}^{\text{fg, gr}}_{0, \mu}(\textbf{U} \mathfrak{g}) \ar[d]^{Forg}\ar[r]_{\cong}& \text{DGCoh}^{\gr}(\widetilde{\fg}_{\calP}\times_{\fg}\{0\})\ar[d]_{Forg}\\
D^b\text{Mod}^{fg}_{0, \mu}(\textbf{U} \mathfrak{g})\ar[r]_{\cong}& \text{DGCoh}(\widetilde{\fg}_{\calP}\times_{\fg}\{0\}) }\]

The composition $\xi\circ\kappa\circ\tilde\gamma^\calP_\mu$ gives the following equivalence
\begin{equation}\label{localization}
D^b\text{Mod}^{\text{fg, gr}}_{\mu,0} (\textbf{U}\fg) \simeq \text{DGCoh}^{\text{gr}}(\tilde{\mathfrak{g}}_{\mathcal{P}}\times_{\fg}\{0\}) \simeq \text{DGCoh}^{gr}(T^* \mathcal{P})\simeq D^b\Coh^{\mathbb{G}_m}(T^*\calP).
\end{equation}

Let $P\subseteq Q \subseteq G$ be two parabolic subgroups, and $\mu$, $\nu\in\Lambda$ be weights which respectively  lie only on those reflection hyperplanes corresponding to $\calP$ and $\calQ$.  We have the natural map $\tilde \pi_{\calP}^{\calQ}: \widetilde\fg_\calP\times_\fg\{0\} \rightarrow \widetilde\fg_\calQ\times_\fg\{0\}$.  The functors on derived categories of coherent sheaves are denoted by  
\begin{align*} 
R\tilde\pi_{\calP*}^{\calQ}: \text{DGCoh}^{\gr}(\widetilde\fg_\calP\times_\fg\{0\}) \to \text{DGCoh}^{\gr}(\widetilde\fg_\calQ\times_\fg\{0\}) \\ L\tilde\pi_{\calP}^{\calQ*}: \text{DGCoh}^{\gr}(\widetilde\fg_\calQ\times_\fg\{0\}) \to \text{DGCoh}^{\gr}(\widetilde\fg_\calP\times_\fg\{0\}) 
\end{align*} 

Translation functors between the ungraded categories defined as in \eqref{eqn:trans_func} have lifts. More precisely,  we have functors  $$T_{\mu}^{\nu}: \text{Mod}^{\text{fg, gr}}_{\mu, 0}(\textbf{U} \mathfrak{g}) \rightarrow \text{Mod}^{\text{fg, gr}}_{\nu, 0}(\textbf{U} \mathfrak{g}) \qquad T_{\nu}^{\mu}: \text{Mod}^{\text{fg, gr}}_{\nu, 0}(\textbf{U} \mathfrak{g}) \rightarrow \text{Mod}^{\text{fg, gr}}_{\mu, 0}(\textbf{U} \mathfrak{g})$$
satisfying the commutativity conditions
\begin{equation}\label{geomtrans}
T_\mu^\nu\circ\tilde\gamma^\calP_{\mu} \cong \tilde\gamma^\calQ_{\nu}\circ R\tilde\pi_{\calQ*}^{\calP}\hbox{ and }T^\mu_\nu\circ\tilde\gamma^\calQ_{\nu}\cong\tilde\gamma^\calP_{\mu}\circ L\tilde\pi_\calQ^{\calP*}.
\end{equation}
One sees that  \cite[Proof of Proposition~5.4.3]{riche} after forgetting the grading, these correspond to the translation functors \eqref{eqn:trans_func}; In particular, they are well-defined on the abelian categories albeit  a priori only defined on the derived categories.


\subsection{Statement of Theorem~\ref{IntrB}} \label{statementB} 
In what follows, we are primarily interested in the special case of Sections~\ref{richebackground} and \ref{subsec:Localization} when $G=SL_n$, $\mu=\mu_r={-\rho+e_1+\cdots+e_r}$, $r=1,\dots,n$.
The singular weight $\mu_r$ corresponds to the parabolic subgroup $P_r$  stabilizing the flag $$ 0 \subset \textbf{k} \{ e_1, \cdots, e_r \} \subset \textbf{k}^n .$$
The corresponding partial flag variety is $\calP_r=G/P_r=\Gr(r,n)$, the Grassmanian of $r$-dimensional vector spaces in $\textbf{k}^n$. 

Our main result is that the above categorification admits a graded lift, which is equivalent to that constructed by Cautis, Kamnitzer, and Licata \cite{ckl}. Let us start by recalling the set-up used there. Note that: $$ T^* \Gr(r,n) = \{ (V,X)\mid 0 \subset V \subset \textbf{k}^n,\ X\in\End_{\textbf{k}}(\textbf{k}^n), \dim V=r, XV=0, X(\textbf{k}^n) \subset V \} $$ The multiplicative group $\Gm$ of the field $\textbf{k}$ acts on $T^* \Gr(r,n)$ by dilation on the fibers: $t \cdot X = t^2 X$. 
For each $0 \leq r \leq n$, define the bounded derived categories of equivariant coherent sheaves on these varieties: $$D(-n+2r)=D^{b}\text{Coh}_{\Gm}(T^*\Gr(r,n))$$ Let $W \subset T^*\Gr(r,n) \times T^*\Gr(r+1,n)$ be the Lagrangian correspondence \[ \{ (0 \subset V \subsetneq V' \subset \textbf{k}^n) \mid \dim V=r,\ \dim V'=r+1, \\  X V' = 0, X(\textbf{k}^n) \subset V \} \] Let us define the functors $\textbf{E}(-n+2r+1), \textbf{F}(-n+2r+1)$ via the following Fourier-Mukai kernels $\calE(-n+2r+1), \calF(-n+2r+1) \in D^b \text{Coh}_{\Gm}(T^*\Gr(r,n) \times T^*\Gr(r+1,n))$. 
\begin{align} \label{eqn:CKL}
\textbf{E}(-n+2r+1): \; & D(-n+2r) \to D(-n+2r+2) \\ \qquad \calE(-n+2r+1) &= \mathcal{O}_W \otimes \det(V')^{n-2r-1} \det(V)^{-(n-2r-1)} \{ n - r - 1 \}\notag \\
\textbf{F}(-n+2r+1): \; & D(-n+2r+2) \to D(-n+2r) \\ \calF(-n+2r+1) &= \mathcal{O}_W \otimes \det(V')\det(V) \{ r \} \notag
\end{align}
These functors can alternatively be expressed as pull-push functors in the derived sense, similar to Lemma~\ref{lem:derFM}. For example, $\textbf{E}(-n+2r+1)$ corresponds to pulling back using the projection $W \rightarrow T^* \Gr(r,n)$, tensoring with a line bundle, and pushing forward under the projection $W \rightarrow T^* \Gr(r+1,n)$, with $W$ endowed with a structure as the derived fiber product.

Cautis, Kamnitzer, and Licata prove the following \cite[Theorem 2.5]{ckl}. 

\begin{theorem}[\cite{ckl}] \label{thm:CKL}
The above categories $D(-n+2r)$ and the functors $\textbf{E}(-n+2r+1), \textbf{F}(-n+2r+1)$ give a geometric categorical $\mathfrak{sl}_2$-action. In particular, there exists an octuple of natural transforms $(i, \pi, \epsilon_1,\epsilon_2, \eta_1,\eta_2, \hat X, \hat T)$ satisfying the conditions of a strong categorical $\mathfrak{sl}_2$-action. \end{theorem} 

\begin{remark} 
There is a difference between the notation we are following and the one of Cautis, Kamnitzer, and Licata \cite{ckl}.
The notation of Cautis, Kamnitzer, and Licata \cite{ckl} is set up so that $\textbf{D}(-n+2r)=D_{\mathbb{G}_m}(T^*\Gr(n-r,n))$ instead, so the line bundles appearing there are slightly different.  Although Cautis, Kamnitzer, and Licata \cite[Theorem $2.5$]{ckl} state theorems over the field $\textbf{k}= \mathbb{C}$, but the proof works equally well when $\textbf{k}$ is an algebraically closed field with characteristic $p \gg 0$. \end{remark}

Recall from Section \ref{richebackground} that $\text{Mod}^{\text{fg}}_{\mu,0} (\textbf{U}\fg)$ admits a Koszul grading, denoted by $\text{Mod}^{\text{fg, gr}}_{\mu,0} (\textbf{U}\fg)$. The forgetful functor is denoted by $F:\text{Mod}^{\text{fg, gr}}_{\mu,0} (\textbf{U}\fg)\to \text{Mod}^{\text{fg}}_{\mu,0} (\textbf{U}\fg)$. Our main result of this section is as follows; the equivalences needed are essentially those constructed by Riche up to twisting by a line bundle.

\begin{thmbis}{IntrB} \label{theoremb} \begin{enumerate}
\item 
On the categories $\oplus_{r=0}^n\text{Mod}^{\text{fg, gr}}_{\mu_{-n+2r},0} (\textbf{U}\fg)$, there are functors $\textbf{E}_{-n+2r+1}$ and $\textbf{F}_{-n+2r+1}$ together with an octuple of natural transforms $(\eta_1,\eta_2,\epsilon_1,\epsilon_2, \iota, \pi, \hat{T}(r),\hat{X}(r))$, satisfying conditions of a strong $\fs\fl_2$-categorification. 
\item Moreover, the forgetful functor \[F:\oplus_{r=0}^n\text{Mod}^{\text{fg, gr}}_{\mu_{-n+2r},0} (\textbf{U}\fg)\to \oplus_{r=0}^n\text{Mod}^{\text{fg}}_{\mu_{-n+2r},0} (\textbf{U}\fg)\] intertwine this categorification and the one from 
 in Theorem \ref{theorema}.
\item There exist equivalences $\Gamma_r: D^b(\text{Mod}^{\text{fg, gr}}_{\mu_r,0} (\textbf{U}\fg)) \simeq D^b(\text{Coh}_{\Gm}(T^* \Gr(r,n))$ of graded triangulated categories, which intertwines the functors $\textbf{E}(-n+2r+1)$ and $\textbf{F}(-n+2r+1)$ from Theorem \ref{thm:CKL} and the functors $\textbf{E}_{-n+2r+1}$ and $\textbf{F}_{-n+2r+1}$ aforementioned in (1).  \end{enumerate}  
\end{thmbis}
Note that (1) follows directly from (3) thanks to Theorem~\ref{thm:CKL}, which not only proved the functorial relations but also gives the octuple of natural transforms from the geometric setting.

In the rest of this section, we prove Theorem~\ref{theoremb}.

\subsection{Linear Koszul duality on Grassmannians}\label{subsubsec:Koszul}

Let us now revisit Theorem \ref{theoremb}, and see what happens to the graded lifts of the translation functors between the categories $\text{Mod}^{\text{fg, gr}}_{\mu_r, 0}(U \mathfrak{g})$ after transporting them across Riche's localization equivalence \ref{localization}. We will find that, after doing this, one obtains a variant of the categorification constructed by Cautis, Kamnitzer, and Licata; the resulting statement will then follow from their results.

First, we will transport the graded translation functors to the right-hand side of the first equivalence (here $\mathcal{P}_r$ is the parabolic corresponding to the weight $\mu_r$): 
\begin{equation}
\label{eqn:rich}
D^b \text{Mod}^{\text{fg, gr}}_{\mu_r, 0}(\textbf{U} \mathfrak{g}) \simeq \text{DGCoh}^{\text{gr}}(\widetilde{\mathfrak{g}}_{\mathcal{P}_r} \times_{\mathfrak{g}} \{0\})
\end{equation}

Define also the parabolic subgroup $P=P_r\cap P_{r+1}$ and the corresponding partial flag variety $\calP=Fl(r,r+1,n)$. Let $\mu_{r, r+1}$ be the singular weight corresponding to it. We have the following maps:
\begin{align*} 
\widetilde{\mathfrak{g}}_{\mathcal{P}_r}  \overset{b_1}{\leftarrow} (\widetilde{\mathfrak{g}}_{\mathcal{P}_r} \times_{\mathcal{P}_r} \mathcal{P})  \overset{a_1}{\leftarrow} \widetilde{\mathfrak{g}}_{\mathcal{P}}  \overset{a_2}{\rightarrow} (\widetilde{\mathfrak{g}}_{\mathcal{P}_{r+1}} \times_{\mathcal{P}_{r+1}} \mathcal{P})  \overset{b_2}{\rightarrow} \widetilde{\mathfrak{g}}_{\mathcal{P}_{r+1}} 
\end{align*}
Abusing notation, we will use the same symbols to denote the corresponding maps after applying the base change $- \times_\fg\{0\}$ to both sides; note that $b_1 \circ a_1 = \widetilde{\pi}_{\mathcal{P}}^{\mathcal{P}_{r+1}}, b_2 \circ a_2 = \widetilde{\pi}_{\mathcal{P}}^{\mathcal{P}_r}$. 
\begin{lemma}
$T_{\mu_r}^{\mu_{r+1}} = T_{\mu_r}^{\mu} T_{\mu}^{\mu_{r+1}}$ and $T_{\mu_{r+1}}^{\mu_r} = T_{\mu_{r+1}}^{\mu} T_{\mu}^{\mu_r}$. 
\end{lemma}
\begin{proof}
This follows from the same proof as \cite[Proposition~2.2.6]{bmr2}.
\end{proof}
Using Lemma~\ref{geomtrans}, under the equivalence \eqref{eqn:rich} the ``graded lifts of translation functors'' mentioned in Theorem~\ref{theoremb} becomes the following.
\begin{align} \label{eqn:Riche_EF} 
\mathfrak{E}(-n+2r+1): \text{DGCoh}^{\text{gr}}(\widetilde{\mathfrak{g}}_{\mathcal{P}_r} \times_{\mathfrak{g}} \{ 0 \}) &\rightarrow \text{DGCoh}^{\text{gr}}(\widetilde{\mathfrak{g}}_{\mathcal{P}_{r+1}} \times_{\mathfrak{g}} \{ 0 \}) \\ \mathfrak{E}(-n+2r+1) &= b_{2*}a_{2*}a_1^* b_1^* \{-(n-r-1)\}; \notag\\ \mathfrak{F}(-n+2r+1): \text{DGCoh}^{\text{gr}}(\widetilde{\mathfrak{g}}_{\mathcal{P}_{r+1}} \times_{\mathfrak{g}} \{ 0 \}) &\rightarrow \text{DGCoh}^{\text{gr}}(\widetilde{\mathfrak{g}}_{\mathcal{P}_r} \times_{\mathfrak{g}} \{ 0 \}),  \\ \mathfrak{F}(-n+2r+1) &= b_{1*}a_{1*}a_2^* b_2^*\{-r\}\notag. 
\end{align}
Note that we inserted an artificial shifting by $\{ -(n-r-1) \}$ and $\{ -r \}$ (without which Theorem \ref{theoremb} would not be true as stated). Cautis and Koppensteiner \cite{cautis} and Cautis and Kamnitzer \cite{ck16} use these functors in a more general set-up of categorical loop $\mathfrak{sl}_n$-actions; see Remark \ref{refine} for more details. 

Secondly, we calculate the images of the functors $\fE(-n+2r+1)$ and $\fF(-n+2r+1)$ under the Koszul duality maps from Lemma \ref{Lem:Koszul}: $$ \text{DGCoh}^{\text{gr}}(\widetilde{\mathfrak{g}}_{\mathcal{P}_r} \times_{\mathfrak{g}} \{ 0 \}) \simeq \text{DGCoh}^{\gr}(T^*\calP_r) $$ Note that $W$, the Lagrangian correspondence (see Section \ref{statementB}), is the intersection of $T^*{\mathcal{P}_r} \times_{\mathcal{P}_r} \mathcal{P}$ and $T^*{\mathcal{P}_{r+1}} \times_{\mathcal{P}_{r+1}} \mathcal{P}$ inside the ambient space $T^* \mathcal{P}$. We have the following maps. 
\begin{align*} 
T^*\mathcal{P}_r  \overset{\beta_1}{\leftarrow}  T^*{\mathcal{P}_r} \times_{\mathcal{P}_r} \mathcal{P}  \overset{\alpha_1}{\rightarrow} &T^*{\mathcal{P}}  \overset{\alpha_2}{\leftarrow} T^*{\mathcal{P}_{r+1}} \times_{\mathcal{P}_{r+1}} \mathcal{P} \overset{\beta_2}{\rightarrow} T^*{\mathcal{P}_{r+1}}, \\ T^*{\mathcal{P}_r} \times_{\mathcal{P}_r} \mathcal{P} \overset{\gamma_1}{\leftarrow} W \overset{\gamma_2}{\rightarrow} T^*{\mathcal{P}_{r+1}} & \times_{\mathcal{P}_{r+1}} \mathcal{P}, \qquad T^*{\mathcal{P}_r} \overset{p}{\leftarrow} W \overset{q}{\rightarrow} T^*{\mathcal{P}_{r+1}} .
\end{align*}

Let the tautological flag on $\calP$ be denoted by $0\subseteq \mathcal{V}'\subsetneq \mathcal{V} \subseteq \calO^n$.
The following lemma is the key step in the proof.
\begin{lemma}\label{lem:kappa}
After applying the Koszul duality equivalences of Lemma \ref{Lem:Koszul} to the sources and targets of the functors $\fE(-n+2r+1)$ and $\fF(-n+2r+1)$, they correspond to the following Fourier-Mukai transforms:
\begin{align*}  \fE(-n+2r+1) &= q_* \circ -\otimes \det(V)^{n-r}\det(V')^{-n+r+1}[n-r-1]\angl{n-r-1} \circ p^*:\\
&\text{DGCoh}^{\gr}(T^*\calP_r) \to \text{DGCoh}^{\gr}(T^*\calP_{r+1}) ;\\
\fF(-n+2r+1) &= p_* \circ -\otimes \det(V)^{-r}\det(V')^{r+1} [r] \{ r \} \circ q^*:\\ &\text{DGCoh}^{\gr}(T^*\calP_{r+1}) \to \text{DGCoh}^{\gr}(T^*\calP_r) .
  \end{align*} \end{lemma}
  
Note that the Fourier-Mukai transforms are taken in the same sense as in \eqref{eqn:CKL}. More precisely, we consider these Fourier-Mukai kernels as coherent sheaves on the product $T^*\calP_r\times T^*\calP_{r+1}$, which are set-theoretically supported on the subvariety $W$. In particular, the pullback $p^*$ is understood as taking inverse-image to the product $T^*\calP_r\times T^*\calP_{r+1}$; the tensor and pushing forward all take place on the product. The fact that the kernels are supported on $W$ ensures the functors to be of finite cohomological amplitude hence are well-defined on the corresponding categories. See \cite[\S~2.2.1]{ckl} for a detailed explanation. We have the following fact. 
\begin{lemma} \label{lem:derFM}
The Fourier-Mukai transform $\fE(-n+2r+1)$ is equivalent to the functor
$$\beta_{2*} \circ  \alpha_{2}^*\circ \otimes \det(V)^{n-r}\det(V')^{-n+r+1}[n-r-1]\angl{n-r-1} \circ \alpha_{1*}\circ \beta_1^*;$$
The  Fourier-Mukai transform $\fF(-n+2r+1)$  is equivalent to the functor \[\beta_{1*} \circ \alpha_{1}^*\circ \otimes \det(V)^{-r}\det(V')^{r+1} \angl{r}[r] \circ \alpha_{2*}\circ \beta_2^*.\]
\end{lemma}
\begin{proof}
As $\beta_1$ and $\beta_2$ are smooth proper maps, for simplicity we consider $\fE(-n+2r+1)$ and $\fF(-n+2r+1)$ as Fourier-Mukai transforms between $DG\Coh^{\text{gr}}(T^*{\mathcal{P}_r} \times_{\mathcal{P}_r} \mathcal{P})$ and $DG\Coh^{\text{gr}}(T^*{\mathcal{P}_{r+1}} \times_{\mathcal{P}_{r+1}} \mathcal{P})$, keeping in mind that $p = \beta_1 \circ \gamma_1$ and $q = \beta_2 \circ \gamma_2$.

Let $DG\Coh^{\text{gr}}(W)$ be the derived category of $\Gm$-equivariant dg-modules on $W$, endowed with the structure of derived fiber product of $\alpha_1$ and $\alpha_2$. Then, the standard base change of derived schemes \cite[Proposition 3.7.1]{BR} implies that $\alpha_{2}^*\circ \alpha_{1*}=\gamma_{2*}\circ\gamma_1^*$. 
Noting that  $\det(V)^{n-r}\det(V')^{-n+r+1}\angl{n-r-1}[n-r-1]$  is pulled back from $\calP$, by projection formula and base change \cite[\S~3.8]{BR}, we have: 
\begin{align*} \alpha_{2}^*\circ \otimes \det(V)^{n-r}\det(V')^{-n+r+1}\angl{n-r-1}[n-r-1] \circ \alpha_{1*}&= \\ \gamma_{2*}\circ\otimes \det(V)^{n-r}\det(V')^{-n+r+1}&\angl{n-r-1}[n-r-1] \circ\gamma_1^*. \end{align*}

Now to prove the statement about $\fE(-n+2r+1)$, we only need to show that the Fourier-Mukai transform in the derived sense is the same as the transform in the sense described in the paragraph before this lemma. This is standard, see e.g., \cite[Proposition~4.2.1 and Remark~4.2.2]{BR}. More precisely, the kernel $\det(V)^{n-r}\det(V')^{-n+r+1}\angl{n-r-1}[n-r-1]$ as a locally free dg-module on $W$ has only finitely many cohomology all of which are coherent sheaves on $W$, therefore  \cite[Remark~4.2.2]{BR}  as a functor $DG\Coh^{\text{gr}}(T^*{\mathcal{P}_r} \times_{\mathcal{P}_r} \mathcal{P})\to DG\Coh^{\text{gr}}(T^*{\mathcal{P}_{r+1}} \times_{\mathcal{P}_{r+1}} \mathcal{P})$ it is equal to its image in  $D^b\Coh^{\text{gr}}_W((T^*{\mathcal{P}_r} \times_{\mathcal{P}_r} \mathcal{P})\times (T^*{\mathcal{P}_{r+1}} \times_{\mathcal{P}_{r+1}} \mathcal{P}))$, which is the Fourier-Mukai transform in the sense considered here.

Similarly for the statement about $\fF(-n+2r+1)$. 
\end{proof}
  
\begin{proof}[Proof of Lemma~\ref{lem:kappa}]
By \cite[Proposition~2.4.5]{riche}, $b_i^*$ is Koszul dual to $\beta_i^*$ for $i=1,2$. Let us calculate the Koszul dual to $a_1^*$. On $\calP$, we have two vector bundles, $F_1=\calP\times_{\calP_r}T^*\calP_r$ and $F_2=T^*\calP$. 
The natural embedding $F_1\inj F_2$ is $\alpha_1$.
Here $a_1^*$ is Koszul dual \cite[Proposition~4.5.2]{riche} to the functor $$\alpha_{1*}\circ \det(F_1)^{-1}\det(F_2)[n_2-n_1]\angl{2(n_2-n_1)}$$ The tangent bundle of $\calP_r$ is  $\Hom(V',{\textbf{k}}^n/V')$. Hence, $\det(T^*\calP_r)=\det(\Hom(\textbf{k}^n/V',V'))$. Recall that: $$ \text{det}(\Hom( \mathcal{V}, \mathcal{W}))= \text{det}(\mathcal{V})^{-\text{rk}(W)} \text{det}(\mathcal{W}) ^{\text{rk}(V)} $$ Hence, $\det(F_1)=\det(\textbf{k}^n/V')^{-r}\det(V')^{n-r}$, and $n_1=r(n-r)$. Similarly, the tangent bundle of $\calP$ is filtered by $\Hom(V',V/V')$, $\Hom(V',\textbf{k}^n/V)$, and $\Hom(V/V',\textbf{k}^n/V)$, hence its determinantal bundle is the tensor product of that of all the three subquotients. Therefore, $\det(F_2)=\det(\textbf{k}^n/V)^{-1}\det(V/V')^{n-(r+1)}$ and $n_2=r(n-r)+n-r-1$. Therefore, $\det(F_2)\det(F_1)^{-1}=\det(V)^{n-r}\det(V')^{-n+r+1}$, and $n_2-n_1=n-r-1$. Putting this together, the functor $b_{2*}a_{2*}a_1^* b_1^*$ corresponds to the following map, after applying the Koszul duality equivalence from Lemma \ref{Lem:Koszul}: $$ \beta_{2*} \circ  \alpha_{2}^*\circ \otimes \det(V)^{n-r}\det(V')^{-n+r+1}[n-r-1]\angl{2(n-r-1)} \circ \alpha_{1*}\circ \beta_1^*.$$

Since the Koszul duality $\kappa$ in Lemma~\ref{Lem:Koszul} commutes with internal shifts  \cite[Remark~11.10]{Rich2}, the functor $\mathfrak{E}$ corresponds to the following map: $$\beta_{2*} \circ  \alpha_{2}^*\circ \otimes \det(V)^{n-r}\det(V')^{-n+r+1}[n-r-1]\angl{n-r-1} \circ \alpha_{1*}\circ \beta_1^*.$$
Using Lemma~\ref{lem:derFM}, the stated formula for $\fE$ follows.

The corresponding statement for the functor $\mathfrak{F}$ is proven analogously. On $\calP$, there are two vector bundles $F_1'=\calP\times_{\calP_{r+1}}T^*\calP_{r+1}$ and $F_2'=F_2=T^*\calP$. The natural embedding $F_1'\inj F_2'$ is denoted by $\alpha'$. We have $\det(F_1')=\det(\textbf{k}^n/V)^{-(r+1)}\det(V)^{n-(r+1)}$ and $\det(F_2')=\det(\textbf{k}^n/V)^{-1}\det(V/V')^{n-(r+1)}$. It follows that $\det(F_2')\det(F_1')^{-1}=\det(V)^{-r}\det(V')^{r+1}$, and $n'_2-n_1'=r$. Putting this together, it follows that the map $b_{1*}a_{1*}a_2^* b_2^*$ corresponds to the following map, under the Koszul duality equivalence from Lemma \ref{Lem:Koszul}:
$$ \beta_{1*} \circ \alpha_{1}^*\circ \otimes \det(V)^{-r}\det(V')^{r+1} [r] \{ 2r \} \circ \alpha_{2*} \circ \beta_2^*. $$
The stated formula for $\fF$ now follows analogously from Lemma~\ref{lem:derFM}, keeping in mind that Koszul duality $\kappa$ commutes with internal shifts. 
\end{proof} 

Recall from Section \ref{richebackground} the equivalences 
$$\xi_i: \text{DGCoh}^{\gr}(T^*\calP_i)\cong D^b \Coh_{\Gm}(T^*\calP_i)$$ 
$0\leq i\leq n$, induced by the regrading sending $\calM_q^p$ to $\calM_q^{p-q}$ (see, e.g., \cite[(1.1.2)]{Rich2}). 
The below lemma follows from the definitions, keeping in mind the descriptions of $\mathfrak{E}(-n+2r+1), \mathfrak{F}(-n+2r+1)$ given in Lemma \ref{lem:kappa}. See the below diagram for an illustration. 
\begin{lemma} \label{lem:xi} 
Under the equivalences $\xi_i$ for $i=r,r+1$, the functors $\mathfrak{E}(-n+2r+1)$ and $\mathfrak{F}(-n+2r+1)$ correspond to the Fourier-Mukai transforms between the $D^b\Coh_{\Gm}(T^*\calP_r)$'s: 
\begin{align*} 
\mathfrak{E}(-n+2r+1)&=q_* \circ \otimes \det(V)^{n-r}\det(V')^{-n+r+1}\angl{n-r-1} \circ p^*\\ \mathfrak{F}(-n+2r+1) &= p_* \circ \otimes \det(V)^{-r}\det(V')^{r+1} \angl{r} \circ q^*. 
\end{align*} 
\end{lemma}
\[\xymatrix{
\text{DGCoh}^{\gr}(T^*\calP_r)\ar[r]\ar[d]^{\xi_r}& \text{DGCoh}^{\gr}(T^*\calP_{r+1})\ar[d]^{\xi_{r+1}}\\
D^b\Coh_{\Gm}(T^*\calP_r)\ar[r]&D^b\Coh_{\Gm}(T^*\calP_{r+1});
} \]

\[\xymatrix{
\text{DGCoh}^{\gr}(T^*\calP_{r+1})\ar[r]\ar[d]^{\xi_{r+1}}& \text{DGCoh}^{\gr}(T^*\calP_r)\ar[d]^{\xi_r}\\
D^b\Coh_{\Gm}(T^*\calP_{r+1})\ar[r]&D^b\Coh_{\Gm}(T^*\calP_r).
}\] 

\subsection{Equivalence to a geometric categorification}
The following lemma is straightforward, whose proof we leave to the reader. 
\begin{lemma}\label{lem:Theta} 
Let $\mathcal{D}(r)$ with $r\in \bbZ$ be categories, and $E(r): \mathcal{D}(r-1) \rightarrow \mathcal{D}(r+1), F(r): \mathcal{D}(r+1) \rightarrow \mathcal{D}(r-1)$ be functors. 
Suppose that we have an octuple of natural transforms $(\eta_1,\eta_2,\epsilon_1,\epsilon_2, \iota, \pi, \hat{T}(r),\hat{X}(r))$ satisfying conditions of a strong  $\mathfrak{sl}_2$-categorification. Assume furthermore that we have functors $\overline{E}(r), \overline{F}(r)$ and  automorphism $\Theta_r$ of the category $\mathcal{D}(r)$ for each $r$ such that $$ \overline{E}(r) = \Theta_{r+1} \circ E(r) \circ \Theta_{r-1}^{-1}, \overline{F}(r) = \Theta_{r-1} \circ F(r) \circ \Theta_{r+1}^{-1}.$$ Then, the functors $\overline{E}(r)$ and $\overline{F}(r)$ satisfies functorial relations of a weak $\mathfrak{sl}_2$-categorification; moreover, there is an octuple of natural transforms $(\eta_1',\eta_2',\epsilon_1',\epsilon_2', \iota', \pi', \hat{T}'(r),\hat{X}'(r))$ satisfying the conditions of a strong  $\mathfrak{sl}_2$-categorification. 
\end{lemma} 

We obtain the following. \begin{lemma} \label{lem:Theta_Grass}
Define $\Theta_r:D^b\Coh_{\Gm}(T^*\Gr(r,N))\to D^b\Coh_{\Gm}(T^*\Gr(r,N))$ to be $\otimes \det(\calV_r)^{r}$ where $\calV_r$ is the tautological subbundle on $\Gr(r,N)$. Then, the $\Theta_r$'s intertwine the pairs of functors $\{ \textbf{E}(-n+2r+1), \textbf{F}(-n+2r+1) \}$ and $\{ \fE(-n+2r+1), \fF(-n+2r+1) \}$.
\end{lemma}
\begin{proof}
In Lemma~\ref{lem:Theta}, take the categorification $\{E_r,F_r\}$ to be the CKL as in \S~4.2. One easily verifies that $\{\overline E_r,\overline F_r\}$ are precisely $\fE(-n+2r+1), \fF(-n+2r+1)$. To see $\Theta_r$'s are induced by Fourier-Mukai transforms, note that they are induced by kernels the same line bundles on the diagonals $T^*\Gr(r,N)\inj T^*\Gr(r,N)\times T^*\Gr(r,N)$. \end{proof}

Summarizing the results of the present section, we obtain  the following. 
\begin{prop}\label{prop:summar} The categories $D(-n+2r) = D^b \text{Coh}_{\Gm}(T^* \mathcal{P}_r)$, the Fourier-Mukai transforms $\fE(-n+2r+1):D(-n+2r)\to D(-n+2r+1)$, $\fF(-n+2r+1):D(-n+2r+1)\to D(-n+2r)$, along with suitable morphisms $i, \pi, \epsilon, \eta, \hat X, \hat T$, constitute a strong categorical  $\mathfrak{sl}_2$-action. \end{prop}
\begin{proof}
This follows from Lemmas~\ref{lem:Theta} and \ref{lem:Theta_Grass}, along with Theorem~\ref{thm:CKL}\cite{ckl}. \end{proof}

\begin{remark} \label{xt} There exist morphisms between  Fourier-Mukai kernels $\hat X(r)$ and $\hat T$ \cite[Theorem~5.1]{ckl}. Moreover, the choice of $\hat X$ and $\hat T$ is parametrized by a certain space $V(1)^{tr} \times V(2)^{tr} \times \textbf{k}^{\times}$  \cite[p.g.20]{ckl}, while the choice of $i, \pi, \epsilon, \eta$ is unique up to scaling by $\textbf{k}^{\times}$.  
\end{remark} 
\subsection{Concluding remarks}\label{subsec:compar}

To summarize the precise relation between the categorification via modular representations $\text{Mod}_{0, \mu}(\textbf{U} \mathfrak{g})$ and the graded translation functors from Lemma~\ref{geomtrans}, 
and the  categorification from Proposition~\ref{prop:summar}, we have the following diagram
\[\xymatrix{
D^b\text{Mod}^{\text{gr}}_{0, \mu}(\textbf{U} \mathfrak{g}) \ar[d]^{Forg}\ar[r]^{\tilde\gamma_\calP}_{\cong}& \text{DGCoh}^{\gr}(\widetilde{\fg}_{\calP}\times_{\fg}\{0\})\ar[d]_{Forg}\ar[r]^{\Theta\circ \xi\circ\kappa}_{\cong}&D^b\Coh_{\Gm}(T^*\calP)\\
D^b\text{Mod}_{0, \mu}(\textbf{U} \mathfrak{g})\ar[r]^{\hat\gamma_\calP}_{\cong}& \text{DGCoh}(\widetilde{\fg}_{\calP}\times_{\fg}\{0\})
}.\]
The equivalence  $\Gamma_r$ claimed in Theorem~\ref{theoremb}(2) is given by the top row of this diagram. 
This finishes the proof of Theorem~\ref{theoremb}. We deduce that: 

\begin{corollary} \label{refinement} When the characteristic of the field is either $0$, or $p > n$, the categorification of Theorem \ref{thm:CKL} \cite{ckl} admits an abelian refinement. \end{corollary} 

\begin{proof} The statement is a direct corollary of Theorem \ref{theoremb} when $p>n$. When the characteristic is $0$, recall that we have a family of exotic $t$-structures on $\text{DGCoh}^{\text{gr}}(\widetilde{\mathfrak{g}}_{\mathcal{P}} \times_{\mathfrak{g}} \{ 0 \})$ indexed by alcoves (cf. Remark 1.5.4 of \cite{bm}). We wish to choose alcoves such that the functors $\mathfrak{E}(-n+2r+1)$ and $\mathfrak{F}(-n+2r+1)$ are exact with respect to the corresponding t-structures. Since the alcove diagram does not depend on the characteristic (cf. Section 1.8 of \cite{bm}), and such a choice of alcoves exists for $p >n$, we can also pick a suitable family of alcoves when the characteristic is $0$ (cf. Theorem 3.0.2 of \cite{bm}). \end{proof}

\begin{remark} \label{refine} Corollary \ref{refinement} fits into a more general framework developed by Cautis and Koppensteiner \cite{cautis}. Uniqueness and existence properties for abelian refinements of the categorification in \cite{ckl} and its variants are proven there.
Up to shifts, the categorification described by equations (\ref{eqn:Riche_EF}) can be extracted from the more general categorical action of $(L\mathfrak{gl}_n, \theta)$ on $\widetilde{\mathcal{K}}_{\mathfrak{g}}$  \cite[the line before Lemma 9.1]{cautis}, and proven by Cautis and Kamnitzer  \cite{ck16}.  The $\mathfrak{sl}_2$-categorification defined by equations \ref{eqn:Riche_EF} admits an abelian refinement \cite[Corollary 9.2]{cautis} using certain exotic t-structures. From the previous section, we deduce that this categorification is equivalent to the categorification  \cite{ckl} by linear Koszul duality. \end{remark} 

\begin{remark} \label{rmk:3.14}
Theorem \ref{theoremb}, together with Corollary~\ref{cor:2.13} further fit into the general frame work of $D$-equivalence in birational geometry. 
In general it has been conjectured by Bondal and Orlov, and proved in dimension 3 by Bridgeland that flops induce derived equivalences. Categories of perverse coherent sheaves were used in Bridgeland's proof, and the resulting equivalence is perverse in the sense of Chuang and Rouquier \cite{CR2}. Later on, Kawamata and Namikawa proved similar results for Mukai-flops between birational symplectic varieties. They further raised the question on whether stratified Mukai flops, in particular those for complementary Grassmannians, induced derived equivalences. 
In \cite{ckl}, an explicit derived equivalence $D^b\Coh(T^*\Gr(k,N)) \cong D^b\Coh(T^*\Gr(N-k,N))$ was achieved as a consequence of the categorification, answering the question of Kawamata and Namikawa, although the question of perversity of this equivalence was still left open. We note that  combining Theorem \ref{theoremb} and Corollary~\ref{cor:2.13}, we obtain that the derived equivalences induced by stratified Mukai flops for complementary Grassmanians are perverse equivalences.
\end{remark}

\begin{remark}
Theorem \ref{theoremb} implies Theorem \ref{theorema}. Indeed, the categorification from Theorem~\ref{theorema} is induced by Fourier-Mukai kernels  \cite[Proposition~5.4.3]{riche} in $DGCoh(\widetilde\fg_{\calP_r}\times_\fg\widetilde\fg_{\calP_{r+1}}\times_\fg\{0\})$, which admits natural graded lifting given by \eqref{eqn:Riche_EF}.  
From Section~\ref{subsubsec:Koszul}, Theorem \ref{theoremb} can be reformulated in terms of Fourier-Mukai kernels on the spaces $\widetilde\fg_{\calP_r} \times_\fg\widetilde\fg_{\calP_{r+1}}\times_\fg\{0\}$.
The natural transforms between the functors are constructed from morphisms between the corresponding Fourier-Mukai kernels \cite[Theorem~5.1]{ckl} (see also Remark~\ref{xt}). 
The nil-affine Hecke relations amount to certain relations between the morphisms of  Fourier-Mukai kernels. 
In particular,  the necessary relations ({\it i.e.}, those from \ref{strongsl2}) between the functors and morphisms on the graded categories of coherent sheaves are satisfied. 
Therefore, the relations between the functors and natural transforms on the ungraded categories of coherent sheaves also hold. 
Here note that the said relations from Definition \ref{2LieAlgebra} imply those from Definition \ref{strongsl2} \cite[Theorem~3.19]{Rou}. \end{remark} 

\begin{example} When $n=2$, Proposition \ref{thm:CKL} can be described as follows. We are categorifying the action of $U_q(\mathfrak{sl}_2)$ on $V^{\otimes 2}$.  We have the following three categories, categorifying the three weight spaces
\[
D(-2) = D_{\mathbb{G}_m}(T^* \mathbb{G}(0,2)) = D_{\mathbb{G}_m}(\text{pt}),  \ \  D(0)=D_{\mathbb{G}_m}(T^* \mathbb{G}(1,2)) = D_{\mathbb{G}_m}(T^*\mathbb{P}^1),  \ \  D(2)=D_{\mathbb{G}_m}(\text{pt}) \]
Below we describe the functors between them by writing down the corresponding Fourier-Mukai kernel. 
 \begin{align*} 
\mathfrak{E}(-1): D(-2) \rightarrow D(0) &\longleftrightarrow \mathcal{O}(1) \{  1\} \\ \mathfrak{F}(-1): D(0) \rightarrow D(-2) &\longleftrightarrow \mathcal{O}(-1) \\ \qquad \mathfrak{E}(1): D(0) \rightarrow D(2) &\longleftrightarrow \mathcal{O}(-1) \\ \mathfrak{F}(1): D(2) \rightarrow D(0) &\longleftrightarrow \mathcal{O}(1) \{ 1 \} 
\end{align*} 
In this case, $D(-2) = D(2)$, $\mathfrak{E}(-1) \simeq \mathfrak{F}(1)$ and $\mathfrak{E}(1) \simeq \mathfrak{F}(-1)$. 

The categorical relations that these functors satisfy are as follows (see also   \cite[pg. 5]{ckl}). 
\begin{align*} \mathfrak{E}(1) \mathfrak{F}(1) \simeq \mathfrak{F}(-1)&\mathfrak{E}(1) \simeq  [1]\{- 1\} \oplus [-1] \{ 1 \} \\ \mathfrak{E}(-1) \mathfrak{F}(-1) &\simeq \mathfrak{F}(1) \mathfrak{E}(1) \qquad \end{align*} 
The second relation is self-evident. The first relation can be verified  directly using Koszul resolution \cite[\S~5.4]{CG} and cohomology of sheaves on $\bbP^1$ \cite[\S~3.5]{Hart}. 


\end{example}
 
\section{Further directions.} \label{further}

In the sequel \cite{nzh2}, we extend the results of the present paper in two directions. Firstly we construct categorical $\mathfrak{sl}_k$-actions by considering a larger collection of singular blocks of modular representations of $\fs\fl_n$, following Sussan \cite{suss}. In this setting, we generalise Theorem B by showing that these categorifications admit a graded lift which is equivalent to a geometric construction of Cautis, Kamnitzer and Licata. Secondly we consider representation categories with non-zero Frobenius central characters, and use their singular blocks to categorify tensor products of symmetric powers of the standard $\mathfrak{sl}_k$-module. We also show that the geometric construction of categorical symmetric Howe duality by Cautis and Kamnitzer \cite{ck16} can be used to obtain a graded lift of this categorification.

One crucial difference in the positive characteristic setting compared to that of \cite{suss} is that it is necessary to look at categories of representations equipped with a torus grading. More precisely, using the notation from Section \ref{modularrep}, we define the category $(\mathfrak{g}, T)-\text{mod}$ as follows: an object consists of a module $M$ with a grading $M = \oplus_{\nu \in X} M_{\nu}$, such that each root vector $E_{\alpha}$ maps $M_{\nu}$ onto $M_{\nu + \alpha}$ and every $H \in \mathfrak{h}$ acts on $M_{\nu}$ as multiplication by $\nu(H)$ (here $X$ is the group of characters of $T$). This torus grading is essential for the formulation of Lusztig's conjectures  for representations of Lie algebras in positive characteristic; see Section 3 of \cite{fiebig} for a precise formulation. 

\newcommand{\arxiv}[1] {\texttt{\href{http://arxiv.org/abs/#1}{arXiv:#1}}} \newcommand{\doi}[1] {\texttt{\href{http://dx.doi.org/#1}{doi:#1}}}

\end{document}